\documentclass{siamonline220329}
\usepackage{amsmath}
\usepackage{amsfonts}
\usepackage{array}
\usepackage{textcomp}
\usepackage{stfloats}
\usepackage{url}
\usepackage{verbatim}
\usepackage{graphicx}
\usepackage{xcolor}
\usepackage{cite}
\usepackage{balance}
\usepackage{algorithm}
\usepackage{algorithmicx}
\usepackage{algpseudocode}
\usepackage{enumitem}
\usepackage{subcaption}
\DeclareMathOperator*{\argmin}{arg\,min}
\newtheorem{remark}[theorem]{Remark}

\author{Arttu Arjas\thanks{Research Unit of Mathematical Sciences, University of Oulu, Finland.
  \email{arttu.arjas@oulu.fi} }
\and Mikko J. Sillanp\"a\"a\footnotemark[2]
\and Andreas Hauptmann\thanks{Research Unit of Mathematical Sciences, University of Oulu and Department of Computer Science, University College London, UK.}}

\date{\today}
\title{Sequential model correction for nonlinear inverse problems\thanks{Submitted to the editors XYZ.
\funding{This work was funded by the Academy of Finland, under projects 338408, 346574, 326291 and the Centre of Excellence of Inverse Modelling and Imaging project 353093.}}}

\begin{document}

\maketitle

\begin{abstract}
Inverse problems are in many cases solved with optimization techniques. When the underlying model is linear, first-order gradient methods are usually sufficient. With nonlinear models, due to nonconvexity, one must often resort to second-order methods that are computationally more expensive. In this work we aim to approximate a nonlinear model with a linear one and correct the resulting approximation error. We develop a sequential method that iteratively solves a linear inverse problem and updates the approximation error by evaluating it at the new solution. This treatment convexifies the problem and allows us to benefit from established convex optimization methods. We separately consider cases where the approximation is fixed over iterations and where the approximation is adaptive. In the fixed case we show theoretically under what assumptions the sequence converges. In the adaptive case, particularly considering the special case of approximation by first-order Taylor expansion, we show that with certain assumptions the sequence converges to a critical point of the original nonconvex functional. Furthermore, we show that with quadratic objective functions the sequence corresponds to the Gauss-Newton method. Finally, we showcase numerical results superior to the conventional model correction method. We also show, that a fixed approximation can provide competitive results with considerable computational speed-up.
\end{abstract}

% REQUIRED

% REQUIRED

\section{Introduction}
Inverse problems appear in numerous places in mathematics, engineering and medicine \cite{muellersiltaneninvprob,arridgedatadriven}. They deal with deducing cause from the observed effects (data). Mathematically, the problem is often written in the form
\begin{equation}
    y = A(x),
\end{equation}
where $y \in Y$ is the measured data and $x \in X$ the unknown cause. Here $A: D(X) \to Y$, where $D(X) \subset X$, is a (possibly) nonlinear model that describes how the measurements are linked to the cause. We assume that both $X$ and $Y$ are infinite-dimensional Hilbert spaces. The inverse problem is then often solved by finding a cause that best matches the data when inserted into the model. Since the problem is usually ill-posed, one has to constrain the solution through regularization. Mathematically this is a variational optimization problem of the form
\begin{equation}
    x^* = \argmin_{u \in X} \left\{F(A(u), y) + \lambda R(u)\right\}
\end{equation}
where the minimized objective function is a sum of functions $F: Y \times Y \to \mathbb{R}_+$ and $R: X \to \mathbb{R}_+$. The first part measures the mismatch between the data and the model output (data fidelity), and the second part regularizes the solution. (Here, $\lambda>0$ is a regularization parameter to be adjusted for each data separately.) In this work we assume that both parts are given by convex functions, e.g. a norm. Given these assumptions, if the model is linear, the resulting objective function is also convex. This means that it has a unique minimizer which can be efficiently found with first-order optimization techniques such as gradient descent or primal-dual methods which are usually computationally cheap to implement \cite{benning2018modern}. In this case the Fréchet derivative of the forward model appearing in the gradient of the objective function is independent of the input of the model, and thus computationally cheap to evaluate repeatedly. However, if the model is nonlinear, the Fréchet derivative depends on the input of the model and must be recomputed for every input. This becomes computationally inefficient particularly when the model is given in functional form and differentiation must be carried out by using numerical methods. Furthermore if the model is nonlinear, the objective function is generally nonconvex with multiple local minima. In this case first-order methods usually converge slowly since the gradient only carries information about the local steepness of the objective function.
Thus, the Gauss-Newton method is often used for solving nonlinear inverse problems. It is a modification of Newton's method that utilizes linearization to avoid computing higher than first-order derivatives. Recently, extensions of primal-dual methods allowing nonlinear models have also been proposed \cite{chen2021nonconvex,valkonen2014primal}, which are also based on linearization of the model.

The purpose of this work is to solve the nonlinear inverse problem with a linear approximation, 
unlocking the benefits of first-order convex optimization. As mentioned in the earlier paragraph, the benefits include faster convergence and speed-up in the computation. 
The approximation naturally creates an approximation error that we have to account for. In this work, we assume that we have access to the accurate model to evaluate the approximation error locally when needed.

Compensations of approximation errors have been extensively studied in the literature. For instance, modeling errors can be efficiently corrected with the established approximation error method (AEM) \cite{arridge2006approximation,kaipio2006statistical} and has been successfully applied in a wide area of inverse problems \cite{hanninen2020application,sahlstrom2020modeling,nissinen2010compensation,tarvainen2009approximation,candiani2021approximation}. It is a linear correction that assumes that the approximation error has a Gaussian distribution. The mean vector and covariance matrix of the distribution must be estimated as a part of solving the inverse problem, for instance from training data under knowledge of the accurate forward operator. 
Naturally, AEM has trouble dealing with non-Gaussian approximation errors, which we also illustrate. To overcome this limitation, some recent works have proposed neural networks for model correction \cite{smylnongaussian, opcor2021lunz,mozumder2021model,hauptmann2018approximate,koponen2021model}, in which case the correction is nonlinear. Neural networks succeed in the task but require training data that represent the solution space well enough. 
A related problem is model approximation using so-called surrogates \cite{liangsurrogate, gpsurrogate}. A surrogate model is an approximative model used often when the exact model is too slow to evaluate. The surrogate, modeled for example by neural networks or Gaussian processes, is trained with pairs of inputs and outputs of the true model and can greatly speed up for example Markov chain Monte Carlo (MCMC) computations requiring hundreds of thousands of model evaluations. However, surrogate-based models suffer from the curse of dimensionality as the number of training samples needed grows exponentially with the dimension of the parameter space. Alternatively, when only the approximate operator is known, but an estimate on the error exists, one can utilize the regularizing sequential subspace optimization method \cite{schopfer2008fast,nitzsche2022compensating,blanke2020inverse} to compensate for the additional error. Another approach has been proposed in \cite{korolev2018image,bungert2020variational} using partially ordered spaces, where upper and lower bounds on the accurate operator are available.

We will particularly consider imaging applications in this work. A prototypical example of ill-posed inverse problem in imaging is deconvolution \cite{diffusion_weicker,muellersiltaneninvprob}. The purpose of deconvolution is to restore an image degraded by some kernel. The kernel can be thought of as a filter that integrates local information of the image thereby reducing its quality. Convolution is also closely connected to diffusion processes, namely diffusion can be seen as convolution with a Gaussian kernel \cite{arridge2019diffnet}. If the kernel does not depend on the image, diffusion is a linear operator. When the kernel is let depend on the image, the operator becomes nonlinear. This can be done for example with the Perona-Malik filter, where the diffusion strength depends on the magnitude of the image gradient \cite{peronamalik}. Thus areas of the image with edges are diffused less than smoother areas.

To deal with the aforementioned problems, we propose a sequential model correction method. Starting from an initial point, we iteratively update the approximation error by evaluating it at the current iterate and solve the inverse problem using the approximate model, eventually converging to a solution. The procedure locally linearizes and thus convexifies the variational problem. We note that we assume that we can evaluate the exact forward model and the approximate model for any $x \in X$. This is required for evaluating the approximation error accurately. We investigate two different cases: one where the approximation is fixed over sequence iterations and one where the approximation adapts locally. The summary of the results of this paper is as follows:
\begin{enumerate}[label=(\roman*)]
    \item For the fixed approximation, we derive the conditions needed for the convergence of the sequence.
    \item For the adaptive approximation, we show that under certain conditions on the approximation, taking small enough steps in the sequence always decreases the original objective function. With further regularity assumptions, we show that the sequence converges to a critical point of the original objective function. 
    \item If the approximation is chosen as a first-order Taylor expansion of the accurate model and the data fidelity and regularization terms are both quadratic, we recover the Gauss-Newton method.
\end{enumerate}
Finally we show with examples that the sequential model correction method outperforms AEM with various models, while being computationally efficient. 
While we have to be able to evaluate the accurate forward model, we need only few evaluations. Particularly, we show that a fixed approximation can yield similar quantitative results than the adaptive approximation, with a computational speed-up of up to factor 8.

This paper is organized as follows: in Section \ref{sec:problemstatement} we define notation, state the problem mathematically and introduce the approximation error method. In Section \ref{sec:convergence} we analyze the convergence of the method. The analysis is separated into fixed and adaptive cases. We also show the connection to the classic Gauss-Newton method. In Section \ref{sec:modelsandimplementation} we introduce the models used in numerical experiments and discuss implementation details. The results are shown and discussed in Section \ref{sec:results}. Finally, concluding remarks are given in Section \ref{sec:conclusion}.

\section{Problem statement} \label{sec:problemstatement}
Let us first establish notation and assumptions used throughout this paper unless otherwise stated.
\begin{center}
\bgroup
\def\arraystretch{1.5}
\begin{tabular}{|c|l|} 
 \hline
 Symbol & Meaning \\
 \hline
 $y$ & Measured data \\
 $x$, $u$ & Unknown we want to reconstruct \\
 $X, Y$ & Hilbert spaces. Norm $\|\cdot\|_{X/Y}$, inner product $\langle \cdot, \cdot \rangle_{X/Y}$ \\
 $A: X \to Y$ & True nonlinear forward operator \\
 $\Tilde{A}: X \to Y$ & Linear approximation of $A$ \\
 $F: Y\times Y \to \mathbb{R}_+$ & Convex data fidelity functional \\
 $R: X \to \mathbb{R}_+$ & Convex regularization functional \\
 $\lambda > 0$ & Regularization parameter \\
 \hline
\end{tabular}
\egroup
\end{center}
%\begin{enumerate}
%    \item $y \in Y$ is the measurement data.
%    \item $x$ or $u \in X$ is the unknown quantity we want to reconstruct.
%    \item $A: X \to Y$ is a nonlinear forward operator that encodes the physical model that the measurements arise from.
%    \item $\Tilde{A}: X \to Y$ is a linear approximation of $A$.
%    \item $F: Y \to \mathbb{R}^+$ is a convex data fidelity (likelihood) functional that measures the distance between model output and data.
%    \item $R: X \to \mathbb{R}^+$ is a convex regularization functional that encodes assumptions about the unknown.
%    \item $\lambda > 0$ is a regularization parameter.
%\end{enumerate}
We wish to solve the inverse problem
\begin{equation} \label{eq:inverseproblem}
    y = A(x) + e
\end{equation}
for the unknown $x$, where $e$ is noise. In the variational framework, solving the inverse problem amounts to solving the variational problem
%\begin{equation} \label{eq:nonlinprob}
%    x^* = \argmin_{u \in \mathbb{R}^n}\left\{\frac{1}{2}\|A(u) - y\|_2^2 + \lambda R(u)\right\},
%\end{equation}
\begin{equation} \label{eq:nonlinprob}
    x^* = \argmin_{u \in X}\left\{F(A(u), y) + \lambda R(u)\right\}.
\end{equation}
For solving the variational problem \eqref{eq:nonlinprob} one needs the Fréchet derivative of $A$, as discussed earlier. Since $A$ is nonlinear, the Fréchet derivative depends on the input, and hence must be recomputed for each input which is time consuming particularly in iterative optimization algorithms.  For this reason we wish to approximate $A$ with a linear model $\Tilde{A}$. The Fréchet derivative of a linear model is independent of the input, making it computationally easier to handle. Writing \eqref{eq:inverseproblem} in terms of $\Tilde{A}$ yields
\begin{equation} \label{eq:inverseproblemerror}
    y = \Tilde{A}x + A(x) - \Tilde{A}x + e = \Tilde{A}x + \varepsilon(x) + e.
\end{equation}
We note that here the approximation by linear model creates an approximation error, denoted by $\varepsilon(x)$. Clearly this formulation of the model is still nonlinear, we have just moved the nonlinearity into $\varepsilon(x)$. Let us then assume that we have access to some initial reconstruction $x_0 \in X$. We can then write the model as
%\begin{equation} \label{eq:nonlinprob2}
%    x^* = \argmin_{u \in \mathbb{R}^n}\left\{\frac{1}{2}\|\Tilde{A}u - y + \varepsilon(u)\|_2^2 + \lambda R(u)\right\},
%\end{equation}
%\begin{equation} \label{eq:nonlinprob2}
%    x^* = \argmin_{u \in X}\left\{F(\Tilde{A}u; y - \varepsilon(u)) + \lambda R(u)\right\},
%\end{equation}
\begin{equation}
    y \approx \Tilde{A}x + \varepsilon(x_0) + e,
\end{equation}
which is linear since $x_0$ is known. This also means we can evaluate $\varepsilon(x_0)$. This leads to the convex variational problem
%\begin{equation} \label{eq:linprob}
%    x^* = \argmin_{u \in \mathbb{R}^n}\left\{\frac{1}{2}\|\Tilde{A}u - y + \varepsilon(x_0)\|_2^2 + \lambda R(u)\right\},
%\end{equation}
\begin{equation} \label{eq:linprob}
    x^* = \argmin_{u \in X}\left\{F(\Tilde{A}u, y - \varepsilon(x_0)) + \lambda R(u)\right\},
\end{equation}
which gives us a local reconstruction depending on $x_0$. From here it is natural to expand this construction into a sequence
%\begin{equation} \label{eq:sequence}
%    x_{k+1} = S(x_k) =  \argmin_{u \in \mathbb{R}^n}\left\{\frac{1}{2}\|\Tilde{A}_ku - y + \varepsilon(x_k)\|_2^2 + \lambda R(u)\right\},
%\end{equation}
\begin{equation} \label{eq:sequence}
    x_{k+1} = S(x_k) = \argmin_{u \in X}\left\{F(\Tilde{A}u, y - \varepsilon(x_k)) + \lambda R(u)\right\}.
\end{equation}
We emphasize that updating the sequence, i.e., solving a linearized and thus convex optimization problem can be done efficiently with first-order optimization methods.

In the following we aim to:
\begin{enumerate}[label=(\roman*)]
    \item Analyze the sequence \eqref{eq:sequence} theoretically  from a model correction perspective. We derive conditions needed for the convergence of the sequence. We also draw connections to existing methods for nonconvex optimization, i.e., with certain choices for the approximate model and the data fidelity and regularization functionals, we obtain the Gauss-Newton algorithm.
    \item Show that the sequential model correction method delivers superior results compared to the conventional method.
\end{enumerate}

\subsection{Approximation error method and non-Gaussianity of the approximation error} \label{sec:aem}
In this section we consider a finite-dimensional setting, that is, $X = \mathbb{R}^m$ and $Y = \mathbb{R}^n$. AEM has traditionally been used for model correction in inverse problems. It exploits the normal distribution to integrate the approximation error out of the model. It assumes that $\varepsilon(x) \sim \mathcal{N}(\mu_\varepsilon, \Sigma_\varepsilon)$, where $\mu_\varepsilon$ and $\Sigma_\varepsilon$ are estimated for instance from a training data set $\{x_i\}_{i=1}^N$ \cite{kaipio2006statistical,opcor2021lunz}. The accurate and approximate model are applied to each data point such that $\varepsilon_i = A(x_i) - \Tilde{A}x_i$. Then the estimates for the mean vector and covariance matrix are
\begin{equation}
    \widehat{\mu}_\varepsilon = \frac{1}{N}\sum_{i=1}^N \varepsilon_i, \ \ \widehat{\Sigma}_\varepsilon = \frac{1}{N-1}\sum_{i=1}^N (\varepsilon_i - \widehat{\mu}_\varepsilon)(\varepsilon_i - \widehat{\mu}_\varepsilon)^T.
\end{equation}
Furthermore, the noise $e$ is assumed independently Gaussian with zero-mean and variance $\sigma^2$. Thus we have $\mathrm{cov}(\varepsilon + e) = \Sigma_\varepsilon + \sigma^2I$, assuming mutual independence of the terms. This allows us to write the variational problem as
\begin{equation}
    x^* = \argmin_{u \in \mathbb{R}^m}\left\{\frac{1}{2}\|L^{-1}(\Tilde{A}u - y + \widehat{\mu}_\varepsilon)\|_2^2 + \lambda R(u)\right\},
\end{equation}
where $L$ is the Cholesky factor of $\widehat{\Sigma}_\varepsilon + \sigma^2I$ such that $\widehat{\Sigma}_\varepsilon + \sigma^2I = L^TL$. We note that here we need to assume that the data fidelity is given as a squared $\ell^2$-norm.

AEM specifically assumes that the approximation error has a Gaussian distribution. In this case multiplying the data with the inverse of the Cholesky factor whitens the noise, making it identically Gaussian, which justifies the use of squared $\ell^2$-data fidelity. However, non-Gaussian errors arise especially when trying to correct nonlinear models. This can be seen by assuming a Gaussian distribution for the unknown, i.e., $x \sim \mathcal{N}(m, C)$, and then looking at the distribution of the approximation error. When approximating a linear model with a linear model, the distribution of the approximation error $(A - \Tilde{A})x$ is still Gaussian with mean $(A - \Tilde{A})m$ and covariance $(A - \Tilde{A})C(A - \Tilde{A})^T$. With nonlinear $A$ this is no longer the case since the mean is not a linear function of $m$, meaning it would be different for all $x \in X$. Similarly the covariance would involve the Jacobian of $A$, which is different depending on the point in $X$ we are at.

\section{Convergence of the sequence} \label{sec:convergence}
We split the theoretical analysis into two parts. In the first part we assume that the approximation is fixed, i.e. it does not depend on the sequence iteration index $k$. In the second part we let the approximation depend on $k$. In the fixed case, we state the conditions needed for the convergence of the sequence. For the adaptive case, we consider the special case of local linear approximation at $x_k$ given by the first-order Taylor expansion. In that case we can show that the original nonconvex function decreases at each step of the sequence. The obtained result readily extends to other approximations that provide a descent direction.

\subsection{Fixed approximation} \label{sec:conv_fixed}
In the case of fixed approximation we investigate the convergence conditions of the sequence based on fixed point iterations. We recall that an iterated function is defined by a composition $S \circ \dots \circ S$ for $S: X \to X$. Then, the Banach fixed-point theorem states that if $S$ is contractive, i.e., $K$-Lipschitz with $K < 1$, the iterations will converge to a unique fixed point. We will use this theorem to state general conditions under which the sequence \eqref{eq:sequence} converges. We will start with simple cases to gain intuition into the subject and work our way to more general cases of the functional $F(A(u),y) + \lambda R(u)$. 

We start by looking at a simple invertible linear system and the behavior when we approximate the exact operator with another invertible linear operator.

\begin{theorem}[Linear and invertible operators]
    Let both $A: X \to Y$ and $\Tilde{A}: X \to Y$ be linear and invertible and $R = 0$. Furthermore, let the operator norm of $I - \Tilde{A}^{-1}A$ be smaller than one. Then sequence \eqref{eq:sequence} converges to $x^* = A^{-1}y$.
\end{theorem}
\begin{proof}
    We have $x_{k+1} = S(x_k) = \Tilde{A}^{-1}(y - \varepsilon(x_k))$. Hence for $x_1, x_2 \in X$,
    \begin{equation*}
        \begin{split}
        &\|S(x_1) - S(x_2)\|_X = \|x_1 - \Tilde{A}^{-1}Ax_1 - x_2 + \Tilde{A}^{-1}Ax_2\|_X\\
        &= \|(I - \Tilde{A}^{-1}A)(x_1 - x_2)\|_X\\
        &\leq K\|x_1 - x_2\|_X,
    \end{split}
    \end{equation*}
    where the Lipschitz constant $K$ is the operator norm of $I - \Tilde{A}^{-1}A$. Thus by Banach fixed-point theorem, $S$ is a contraction and admits a unique fixed point. Let $x^*$ be a fixed point of $S$. Then
    \begin{equation}
    \begin{split}
        x^* &= \Tilde{A}^{-1}(y - Ax^* + \Tilde{A}x^*)\\
        \iff x^* &= A^{-1}y
        \end{split}
    \end{equation}
\end{proof}
The result essentially means that if $\Tilde{A}^{-1}A$ is close enough to the identity, that is, $A$ is close to $\Tilde{A}$, the sequence converges to the exact solution of the original problem.

We now examine a case where we need regularization to make the inverse problem uniquely solvable. In particular, we still assume that both $A$ and $\Tilde{A}$ are linear and both the data fidelity and regularizer are given as quadratic functionals. This yields the sequence
\begin{equation}
    x_{k+1} = \argmin_{u \in X}\left\{\frac{1}{2}\|\Tilde{A}u - y + \varepsilon(x_k)\|_Y^2 + \frac{\lambda}{2} \|u\|_X^2\right\}.
\end{equation}
\begin{theorem}[Linear operators with Tikhonov regularization]
    Let both $A: X \to Y$ and $\Tilde{A}: X \to Y$ be linear, $F(\Tilde{A}u, y - \varepsilon(x_k)) = \frac{1}{2}\|\Tilde{A}u - y + \varepsilon(x_k)\|_Y^2$, $R(u) = \frac{1}{2}\|u\|_X^2$ and $T = (\Tilde{A}^*\Tilde{A} + \lambda I)^{-1}\Tilde{A}^*$. Moreover, let the operator norm of $T(A - \Tilde{A})$ be less than one. Then sequence \eqref{eq:sequence} converges to 
    \[
        x^* = \left(I + T(A - \Tilde{A})\right)^{-1}Ty.
    \]
\end{theorem}
\begin{proof}
In this case we have $x_{k+1} = S(x_k) = Ty - T\varepsilon(x_{k})$. Hence for $x_1, x_2 \in X$,
\begin{equation}
    \begin{split}
        \|S(x_1) - S(x_2)\|_X &= \|T\varepsilon(x_1) - T\varepsilon(x_2)\|_X\\
        &= \|TAx_1 - T\Tilde{A}x_1 - TAx_2 + T\Tilde{A}x_2\|_X\\
        &= \|T(A-\Tilde{A})(x_1 - x_2)\|_X\\
        &\leq K\|x_1 - x_2\|_X,
    \end{split}
\end{equation}
where the Lipschitz constant $K$ is the operator norm of $T(A-\Tilde{A})$. Thus, by Banach fixed-point theorem, $S$ is a contraction and admits a unique fixed point. Let $x^*$ be a fixed point of $S$. Then,
\begin{equation}
\begin{split}
   x^* &= Ty - T(A - \Tilde{A})x^*\\ 
    \iff x^* &= \left(I + T(A - \Tilde{A})\right)^{-1}Ty
\end{split}
\end{equation}
\end{proof}
Now the solution is the minimizer of the Tikhonov functional with the approximate operator multiplied by a correction term $(I + T(A - \Tilde{A}))^{-1}$. Clearly the correction term is just identity when $A = \Tilde{A}$.

%The proof of convergence for the nonlinear case is based on the implicit function theorem (IFT).
%\begin{theorem}
%    Let $A$ be a nonlinear and $\Tilde{A}$ a linear map. Let $S(x) = \argmin_{{u \in \mathbb{R}^n}} f(u,x)$, where $f \in C^2(\mathbb{R}^{2n},\mathbb{R})$ is strictly convex over the variable $u$, and $\delta, \eta$ suitably chosen. Then sequence \eqref{eq:modsequence} converges.
%\end{theorem}
%\begin{proof}
%    Since $f$ is strictly convex in $u$, it has a unique minimizer $u^* \in \mathbb{R}^n$. Thus we have $\nabla_u f(u^*, x) = 0$ and the Jacobian matrix $J_{\nabla_uf}(u^*, x)$ is invertible. Then by IFT, there exists an open set $U \in \mathbb{R}^n$ in the neighborhood of $u^*$ and a continuously differentiable function $g: U \to \mathbb{R}^n$ such that $\nabla_u f(g(x), x) = 0$. Since the gradient vanishes at $g(x)$, by strict convexity, it is a unique minimizer of $f$. It follows directly that $S(x) = g(x)$. Since $S(x)$ is continuously differentiable, it is also Lipschitz. Hence sequence \eqref{eq:modsequence} converges.
%\end{proof}

Let us make a few remarks. In the previous cases we required a linear operator to be contractive. This can be verified for example through power method \cite{matrixcomp_golubvanloan}. For nonlinear operators it is harder to verify. However, there are three heuristics we can use to enforce contractivity in that case as well. 
\begin{enumerate}[label=(\roman*)]
    \item We note that in the linear cases the Lipschitz constant (largest singular value) depends on $A - \Tilde{A}$. The better the approximation is, the more likely the sequence converges.
    \item  In the Tikhonov regularized case, the Lipschitz constant depends on $T$, which in turn depends on the regularization parameter $\lambda$. In particular, increasing $\lambda$ decreases the Lipschitz constant. This is because the singular values of $\Tilde{A}^*\Tilde{A} + \lambda I$ increase with $\lambda$, while the singular values of $(\Tilde{A}^*\Tilde{A} + \lambda I)^{-1}$ decrease with $\lambda$.
    \item  We can define a damped sequence as
    \begin{equation}
        x_{k+1} = S^\delta(x_k) = \delta_k S(x_k) + (1 - \delta_k) x_k,
    \end{equation}
    where $\delta \in \mathbb{R}$ and $S$ is defined in \eqref{eq:sequence}. The damped sequence has the same fixed point as the original sequence, if it exists. However, now the Lipschitz constant of the update rule depends on $\delta$. In theory, it is possible to find a value of $\delta$ that minimizes the Lipschitz constant at the fixed point. This trick is used for example in the Babylonian method for computing square roots \cite{sqrt2}. Using this method requires knowledge of the solution which is not available in practice. It is still possible to experiment with different values and examine the sequence's behaviour.
\end{enumerate}

In the nonlinear case with a general regularizer it is harder to make a general statement. The Lipschitz properties of the minimization operator $S$ depend on the properties of the objective function.
To ensure convergence in practice however, one could evaluate the original objective function $F(A(u), y) + \lambda R(u)$ at each iterate $x_k$ and terminate the sequence when the objective function can no longer be decreased.

\subsection{Adaptive approximation}
We may also let the approximation $\Tilde{A}$  depend on $k$, i.e. we let the approximation $\Tilde{A}_k$ change at every step of the sequence. Then the sequence is given in general form as
\begin{equation} \label{eq:adaptivesequence}
    x_{k+1} =  \delta_k \argmin_{u \in X}\left\{F(\Tilde{A}_k u, y - \varepsilon(x_k)) + \lambda R(u)\right\} + (1 - \delta_k) x_k,
\end{equation}
with step size $\delta_k$. This case can not be analyzed with the fixed-point theory since the mapping changes at each iteration. Here, one could consider a variety of approximations, but we will specifically focus on the case where the approximation is given as a first-order Taylor expansion centered at $x_k$. In that case, $\Tilde{A}_k$ is the Fréchet derivative of $A$ evaluated at $x_k$.

We note, that this choice for the adaptive approximation yields a sequence similar to the successive linearized and regularized Gauss-Newton method, see \cite{wang2021physics,margotti2022range} for recent applications. 
Nevertheless, we emphasize that the successive linearized and regularized Gauss-Newton method is a special case of the sequence formulation \eqref{eq:adaptivesequence}.

\subsubsection{Approximation by Taylor expansion}
Let us now consider local linear approximations by Taylor expansion. That is, the approximation is of the form $\Bar{A}x = A(x') + J(x')(x - x')$, where $x' \in X$ is given and $J(x)$ is the Fréchet derivative of $A$ evaluated at $x$. A convenient choice is to choose $x'$ as the current element of the sequence, yielding
%\begin{equation}
%    x_{k+1} =  \argmin_{u \in \mathbb{R}^n}\left\{\frac{1}{2}\|A(x_k) + J_k(u - x_k) - y\|_2^2 + \lambda R(u)\right\}.
%\end{equation}
\begin{equation} \label{eq:locallinear}
\begin{split}
    x_{k+1} &=  \delta_k \argmin_{u \in X}\left\{F(J_k u, y - A(x_k) + J_k x_k) + \lambda R(u)\right\} + (1 - \delta_k) x_k \\
    &= \delta_k S_k(x_k) + (1 - \delta_k) x_k \\
    &= x_k + \delta_k (S_k(x_k) - x_k) \\
    &= x_k + \delta_k p_k,
\end{split}
\end{equation}
where we denote $J(x_k) = J_k$ for clarity. We write the sequence in this form to more easily see the connection to iterative optimization algorithms such as gradient descent. Here, we also assume that data fidelity functional is of the form $F(a,b) = F(a - b)$, i.e., the mismatch between the data and model output depends on the difference between the two. We note, that this scheme is also related to surrogate approaches for nonconvex optimization, for instance the majorization-minimization algorithm \cite{majormin_lange}, which sequentially constructs surrogates that majorize the nonconvex function and minimizes the surrogates. Likewise, the adaptive sequential method constructs a convex surrogate of the nonconvex function at $x_k$. The difference is that the surrogate does not necessarily majorize the non-convex function. These surrogates are sequentially minimized, decreasing the value of the original nonconvex objective function at each step. This is stated in the following theorem. For clarity of the presentation, we define $L(u) = F(A(u), y) + \lambda R(u)$ and $L_k^s(u) = F(J_k u, y - A(x_k) + J_k x_k) + \lambda R(u)$.

\begin{figure*}[!t]
\centering
\begin{minipage}{.5\textwidth}
  \centering
  \includegraphics[scale=.5]{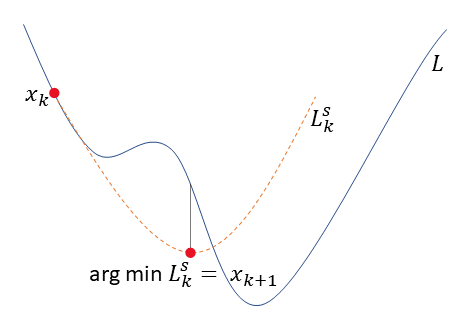}
  %\captionof{figure}{A figure}
\end{minipage}%
\begin{minipage}{.5\textwidth}
  \centering
  \includegraphics[scale=.5]{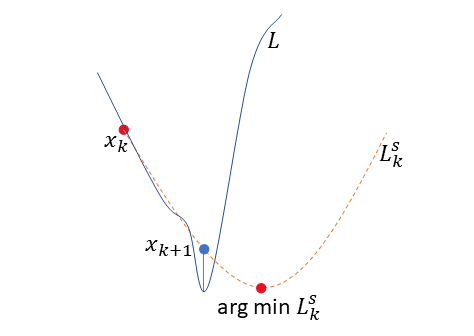}
  %\captionof{figure}{Another figure}
\end{minipage}
\caption{Illustration of the surrogate function $L_k^s$ that results from the local linear approximation. We may be able to get past some local minima, as shown in the picture on the left. The picture on the right shows that not controlling the step size might actually result in increasing the function value.}
\label{fig:surrogate}
\end{figure*}

\begin{theorem}[Descend property] \label{thm:descend}
    Let $A: D(X) \to Y$ be a nonlinear operator, $F: Y\times Y \to \mathbb{R}_+$ a convex data fidelity term and $R: X \to \mathbb{R}_+$ a convex regularization functional. Moreover, let $A$, $F$ and $R$ be differentiable. Then for small enough $\delta_k \in (0, 1]$, the iterates of the sequence \eqref{eq:locallinear} satisfy
    \begin{equation*}
        L(x_{k+1}) \leq L(x_k).
    \end{equation*} %Define a damped version of \eqref{eq:locallinear} as
    %\[
     %   x_{k+1} =  \delta_k \argmin_{u \in X}\left\{L_k^s (u)\right\} + (1 - \delta_k) x_k
    %\]
    %and let $\nabla L(x_k) \neq 0$, i.e., $x_k$ is not a stationary point of $L$. For small enough $\delta_k \in [0,1]$, the sequence converges to a neighborhood of a local minimizer of $L$.
\end{theorem}
\begin{proof}
    Since $J_k$ is linear and $F$ and $R$ are convex, $L_k^s$ is convex. From the definition of convexity, we have
    \begin{equation*}
        \begin{split}
            L_k^s(x_{k+1}) &= L_k^s(\delta_k S_k(x_k) + (1 - \delta_k) x_k) \\
    &\leq \delta_k L_k^s(S_k(x_k)) + (1 - \delta_k)L_k^s(x_k) \\
    &= \delta_k (L_k^s(S_k(x_k)) - L_k^s(x_k)) + L_k^s(x_k) \\
    &\leq L_k^s(x_k)
        \end{split}
    \end{equation*}
for all $\delta_k \in [0, 1]$, since $S_k(x_k)$ minimizes $L_k^s$. From convexity also follows
\begin{equation*}
\begin{split}
    \langle \nabla L_k^s(x_k), x_{k+1} - x_k \rangle_X &\leq L_k^s(x_{k+1}) - L_k^s(x_{k}) \leq 0.   
\end{split}
\end{equation*}
It is left to verify that $\nabla L(x_k) = \nabla L_k^s(x_k)$. We have
\begin{equation*}
    \nabla L_k^s(u) = J_k^* \nabla F(J_k u, y - A(x_k) + J_k x_k) + \lambda \nabla R(u).
\end{equation*}
Evaluating the gradient at $x_k$ gives
\begin{equation*}
\begin{split}
    \nabla L_k^s(x_k) &= J_k^* \nabla F(J_k x_k, y - A(x_k) + J_k x_k) + \lambda \nabla R(x_k) \\
    &= J_k^* \nabla F(J_k x_k - y + A(x_k) - J_k x_k) + \lambda \nabla R(x_k) \\
    &= J_k^* \nabla F(A(x_k) - y) + \lambda \nabla R(x_k) \\
    &= J_k^* \nabla F(A(x_k), y) + \lambda \nabla R(x_k) \\
    &= \nabla L(x_k),    
\end{split}
\end{equation*}
which means
\begin{equation*}
    \langle \nabla L_k^s(x_k), x_{k+1} - x_k \rangle_X = \langle \nabla L(x_k), x_{k+1} - x_k \rangle_X \leq 0.
\end{equation*}
Hence by Taylor's theorem
\begin{equation*}
\begin{split}
    L(x_{k+1}) = L(x_k + \delta_k p_k) = L(x_k) + \langle \nabla L(x_k), \delta_k p_k \rangle_X + \mathcal{O}(\delta_k^2).
\end{split}
\end{equation*}
Since the remainder depends on $\delta_k$ at least quadratically, we can always find small enough $\delta_k$ such that $\langle \nabla L(x_k), \delta_k p_k \rangle_X + \mathcal{O}(\delta_k^2) = \langle \nabla L(x_k), x_{k+1} - x_k \rangle_X + \mathcal{O}(\delta_k^2) \leq 0$, which finally gives $L(x_{k+1}) \leq L(x_k)$.
\end{proof}

We emphasize that the approximative model for the sequence \eqref{eq:adaptivesequence}, does not necessarily need to be the Fréchet derivative of the exact model. It only needs to be chosen such that the value and gradient of the surrogate at the current iterate matches those of the nonconvex functional.
%In this case, we can formulate a corollary with a general condition to obtain similarly local convergence.
\subsubsection{Convergence to a critical point}
In this section we follow \cite{attouch2013convergence} and hence consider the finite dimensional case $Y = \mathbb{R}^n$ and $X = \mathbb{R}^m$. The descent property itself does not guarantee that the sequence converges to a local minimizer or even a critical point of the nonconvex functional. In general, it is possible for the sequence to have multiple limit points \cite{absil2005convergence}. Thus it is necessary to consider functions that possess certain structure. In the nonconvex optimization literature the Kurdyka-\L ojasiewicz (KL) property is often exploited to prove convergence results (see \cite{attouch2013convergence} for the definition), as it is a general property satisfied by many classes of functions. It is used to prove that the trajectory defined by  the sequence has finite length.

In the following we assume that the functional $L$ satisfies the KL-property and has $K$-Lipschitz gradient. For the sequence to converge to a critical point, the following conditions need to hold:
\vspace{.5cm}
\begin{enumerate}[label=(\roman*)]
    \item $\langle \nabla L(x_k), x_{k+1} - x_k \rangle_X + \frac{a}{2}\|x_{k+1} - x_k\|_X^2 \leq 0$,
    \item $\|\nabla L(x_k)\|_X \leq b\|x_{k+1} - x_k\|_X$,
    \item There exists a subsequence $(x_{k_j})_{j \in \mathbb{N}}$ and $x^* \in X$ such that $x_{k_j}\to x^*$ and $L(x_{k_j}) \to L(x^*)$ as $j \to \infty$, 
\end{enumerate}
\vspace{.5cm}
for some positive $a$ and $b$ such that $a > K$. In \cite{attouch2013convergence} (pages 106--107), fulfilling these conditions is shown to result in convergence to a critical point. The following theorem shows what assumptions are needed for the conditions to hold.

\begin{theorem}
    Let $A: D(X) \subset X \to Y$ be a nonlinear operator, $F: Y\times Y \to \mathbb{R}_+$ a convex data fidelity term and $R: X \to \mathbb{R}_+$ a convex regularization functional. Furthermore, let $A$, $F$ and $R$ be chosen such that $L_k^s$ is $m$-strongly convex with $K_s$-Lipschitz gradient and that $L$ has a $K$-Lipschitz gradient such that $m > K$. Let $L$ satisfy the Kurdyka-\L ojasiewicz property. Then for any bounded sequence $(x_k)_{k \in \mathbb{N}}$ generated by Eq. \eqref{eq:locallinear}, conditions (i), (ii) and (iii) hold.
\end{theorem}
\begin{proof}
    (i): Since $L_k^s$ is $m$-strongly convex, we have
    \begin{equation*}
        \langle \nabla L_k^s(x_k), x_{k+1} - x_k \rangle_X \leq L_k^s(x_{k+1}) - L_k^s(x_k) - \frac{m}{2}\|x_{k+1} - x_k\|_X^2.
    \end{equation*}
    We know from Theorem \ref{thm:descend} that $L_k^s(x_{k+1}) - L_k^s(x_k) \leq 0$ and $\nabla L_k^s(x_k) = \nabla L(x_k)$. Thus we conclude that
    \begin{equation*}
        \langle \nabla L(x_k), x_{k+1} - x_k \rangle_X + \frac{m}{2}\|x_{k+1} - x_k\|_X^2 \leq 0.
    \end{equation*}
    (ii): We estimate
    \begin{equation*}
    \begin{split}
        \|\nabla L(x_k)\|_X &= \|\nabla L_k^s(x_k)\|_X \\
        &= \|\nabla L_k^s(x_k) - \nabla L_k^s(S_k(x_k))\|_X \hspace{3cm} (\nabla L_k^s(S_k(x_k)) = 0) \\
        &= \|\nabla L_k^s(x_k) - \nabla L_k^s(\delta_k^{-1}(x_{k+1} - x_k) + x_k)\|_X \hspace{.7cm} (\mathrm{Rearrange \ \eqref{eq:locallinear}}) \\
        &\leq K_s \|x_k - (\delta_k^{-1}(x_{k+1} - x_k) + x_k)\|_X \hspace{2.05cm} (\mathrm{Lipschitz \ gradient}) \\
        &= K_s \|\delta_k^{-1}(x_k - x_{k+1})\|_X \\
        %&= \frac{K_s}{\delta_k}\|x_k - x_{k+1}\|_X \\
        &= \frac{K_s}{\delta_k}\|x_{k+1} - x_{k}\|_X \\
    \end{split}
    \end{equation*}
    (iii): This condition follows from the continuity of $L$ and boundedness of the sequence $(x_k)_{k \in \mathbb{N}}$.
\end{proof}

\begin{remark}
For the convergence proof in \cite{attouch2013convergence} it is required that the strong convexity constant $m$ of the surrogate functional is larger than the Lipschitz constant $K$ of the gradient of the nonconvex functional. Our formulation of the sequence \eqref{eq:adaptivesequence} does not guarantee this. However, when the data fidelity is given by the squared $\ell^2$-norm, the surrogate functional can always be "squeezed" to increase the constant. We only need to make sure that $\nabla L_k^s(x_k) = \nabla L(x_k)$.
\end{remark}

%\begin{corollary}
%    Let $\Tilde{A}_k: X \to Y$ be a linear operator such that the iterates of the sequence \eqref{eq:adaptivesequence} satisfy $\langle \nabla L(x_k), x_{k+1} - x_k \rangle_X < 0$, that is, $x_{k+1} - x_k$ is a descent direction for the functional $L$ at $x_k$. \textcolor{blue}{Then the descent property \ref{thm:descend} holds}.
%\end{corollary}

Figure \ref{fig:surrogate} illustrates the convex surrogate obtained with the local linear approximation. It might let us escape some local minima. However, taking too long steps may end up increasing the value of the function we try to minimize. The optimal step size can be chosen for example by line search, that is, we choose such $\delta$ that minimizes the objective function:
\begin{equation}
    \delta_k = \argmin_{\delta^* \in [0, 1]} L(\delta^* \argmin_{u \in X} \left\{ L_k^s(u) \right\} + (1 - \delta^*) x_k)
\end{equation}

\subsubsection{Connection to the Gauss-Newton method}
The Gauss-Newton method was originally designed for solving nonlinear least-squares problems \cite{wrightoptimization}. It is an iterative optimization method involving linearization of the nonlinear model at every iterate. In inverse problems, it is most often used to solve variational problems of the form
\begin{equation}
    x^* = \argmin_{u \in X}\left\{\frac{1}{2}\|A(u) - y\|_Y^2 + \lambda R(u)\right\},
\end{equation}
where $R$ must be twice continuously differentiable. The updating step is given as \cite{wrightoptimization, Schweigergaussnewton}
\begin{equation}
    x_{k+1} = x_k + \delta_k (J_k^*J_k + \lambda \nabla^2 R_k)^{-1}[J_k^*(y - A(x_k)) - \lambda \nabla R_k],
\end{equation}
where $J_k$ is the Fréchet derivative of $A$ evaluated at $x_k$, $\nabla R_k$ is the gradient and $\nabla^2 R_k$ is the second derivative of $R$ evaluated at $x_k$ and $\delta_k$ is a step size chosen with line search. It turns out that with certain choices the adaptive sequential method is equivalent to the Gauss-Newton algorithm. In particular, we have to always choose $F$ as the $L^2$ norm. Then, if we choose $R = 0$, requiring that $J_k^*J_k$ is invertible, the adaptive sequential update becomes
\begin{equation}
    \begin{split}
    x_{k+1} &= \delta_k (J_k^*J_k)^{-1}J_k^*(y - A(x_k) + J_k x_k) + (1 - \delta_k) x_k \\
    &= x_k + \delta_k (J_k^*J_k)^{-1}J_k^*(y - A(x_k)),
    \end{split}
\end{equation}
which is exactly the same as Gauss-Newton update. Further, if we choose $R(u) = \frac{1}{2}\|u\|_X^2$, the adaptive sequential update is
\begin{equation} \label{eq:adaptive_quad}
    x_{k+1} = \delta_k (J_k^*J_k + \lambda I)^{-1}J_k^*(y - A(x_k) + J_k x_k) + (1 - \delta_k) x_k.
\end{equation}
The connection to Gauss-Newton is a bit more tricky to see here. The Gauss-Newton update can be written as
\begin{equation}
    \begin{split}
    x_{k+1} &= x_k + \delta_k (J_k^*J_k + \lambda I)^{-1}[J_k^*(y - A(x_k)) - \lambda x_k] \\
    &= (I - \lambda \delta_k (J_k^*J_k + \lambda I)^{-1})x_k + \delta_k (J_k^*J_k + \lambda I)^{-1}J_k^*(y - A(x_k)) \\
    &= (1 - \delta_k) x_k + \delta_k J_k^*(J_k J_k^* + \lambda I)^{-1}J_k x_k + \delta_k (J_k^*J_k + \lambda I)^{-1}J_k^*(y - A(x_k)) \\
    &= (1 - \delta_k) x_k + \delta_k (J_k^*J_k + \lambda I)^{-1} J_k^*J_k x_k + \delta_k (J_k^*J_k + \lambda I)^{-1}J_k^*(y - A(x_k)) \\
    &= \delta_k (J_k^*J_k + \lambda I)^{-1} J_k^*(y - A(x_k) + J_k x_k) + (1 - \delta_k) x_k,
    \end{split}
\end{equation}
which is the same as the adaptive sequential update in \eqref{eq:adaptive_quad}. The third row in the above equation follows by Woodbury formula \cite{woodbury}
\begin{equation}
    (J_k^*J_k + \lambda I)^{-1} = \lambda^{-1}I - \lambda^{-2}J_k^*(I + \lambda^{-1} J_k J_k^*)^{-1}J_k
\end{equation}
 and the fourth row by push-through identity \cite{inverseofsum_henderson}
 \begin{equation}
     J_k^*(J_k J_k^* + \lambda I)^{-1} = (J_k^*J_k + \lambda I)^{-1}J_k^*.
 \end{equation}
In general, the adaptive sequential method and Gauss-Newton algorithm are the same for quadratic objective functions. The reason for this is that one step of Gauss-Newton is enough to minimize a quadratic function.

\subsection{Non-differentiable case}
In the previous section we required the data fidelity $F$ and regularization term $R$ to be continuously differentiable functions. However, in principle, the sequential formulation allows the use of any data fidelity and regularization term. In practice, and as we will show in the experiments, we observe good performance for non-differentiable $L^1$ data fidelity and total variation (TV) regularization. This implies that the theory could be modified to allow for non-differentiable objective functions and possibly non-convex data fidelity or regularizer. As these are not the focus of this study, we leave the analysis of other cases for future studies.

\begin{figure}[!t]
    \centering
    \includegraphics[width=\linewidth]{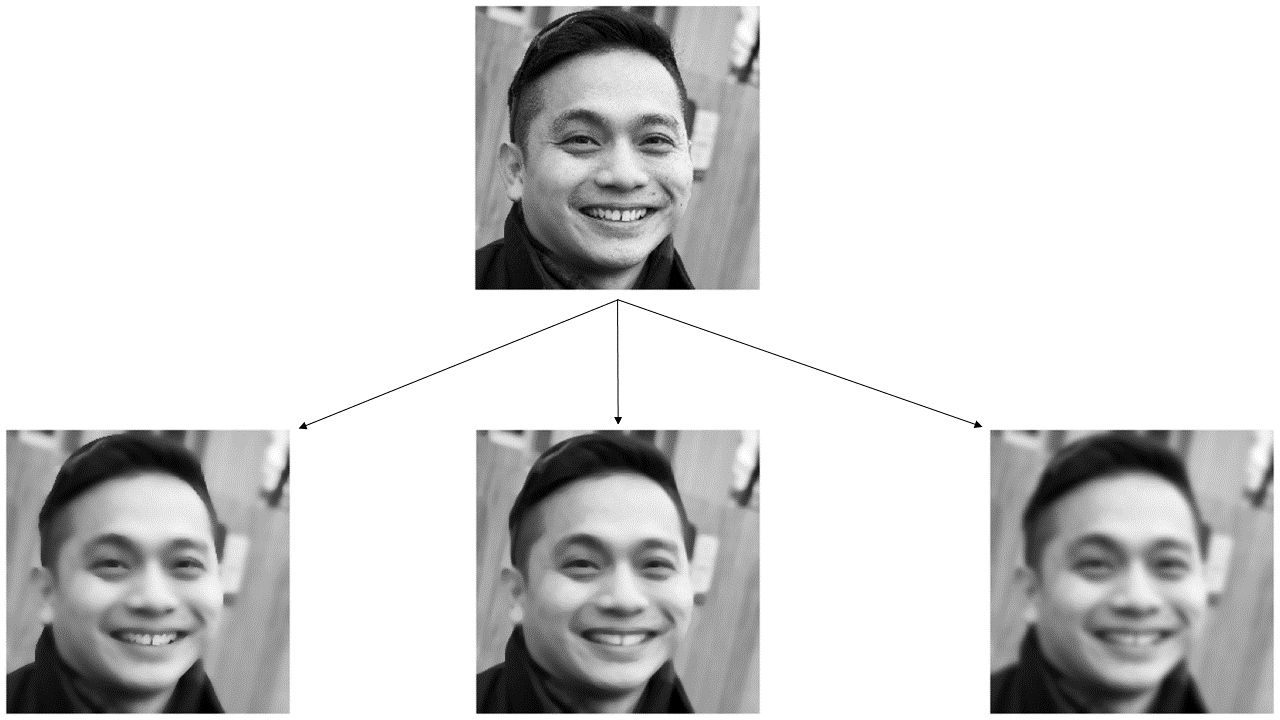}
    \caption{Illustration of how the operators considered in this work distort the ground truth image (top). Nonlinear diffusion (NLD) (bottom left), curvature flow (CF) (bottom middle) and linear diffusion (LD) (bottom right).}
    \label{fig:approx_illustration}
\end{figure}

\section{Models and implementation} \label{sec:modelsandimplementation}
We test the sequential correction method with diffusion-type operators that have been extensively used in imaging. A general diffusion operator $A:x_0\mapsto x_T$ on $\mathbb{R}^d$ for a fixed time interval $(0, T]$ is defined by the partial differential equation 
\begin{equation}
    \begin{cases}
        \partial_t x &= \nabla \cdot (\gamma(x) \nabla x) \\
        x &= x_0 \ \ \mathrm{when} \ t=0,
    \end{cases}
\end{equation}
for diffusivity $\gamma$. If we let $\gamma = 1$, the operator becomes linear and the model corresponds to convolution with a Gaussian kernel. We can also consider a model with varying diffusivity. One possibility is to use the Perona-Malik filter such that $\gamma(x) = (1 + |\nabla x|^2/\kappa^2)^{-1}$, where $\kappa > 0$ is a contrast parameter \cite{peronamalik}. Numerically the operator can be implemented by explicit iterative time-stepping algorithm. Starting from initial value $x_0$, the $k$-th iteration is defined as
\begin{equation}
    x_{k+1} = x_k + \delta t \nabla \cdot \gamma(x_k) \nabla x_k,
\end{equation}
with step size $\delta t$. We also consider a related curvature flow operator. On $\mathbb{R}^d$ and fixed time interval $(0, T]$ the operator is defined as \cite{diffusion_weicker}
\begin{equation}
    \begin{cases}
        \partial_t x &= |\nabla x| \nabla \cdot \left(\frac{\nabla x}{|\nabla x|} \right) \\
        x &= x_0 \ \ \mathrm{when} \ t=0.
    \end{cases}
\end{equation}
Similarly as with the diffusion operator, the curvature flow operator can be numerically implemented with iterative time-stepping algorithm
\begin{equation}
    x_{k+1} = x_k + \delta t (|\nabla x| + \psi) \nabla \cdot \left( \frac{\nabla x_k}{|\nabla x_k| + \psi}\right)
\end{equation}
with step size $\delta t$ and threshold $\psi$ for differentiability and stability by avoiding division with a number too close to zero. The models are referred to as NLD, CF and LD, for nonlinear diffusion, curvature flow and linear diffusion, respectively. Figure \ref{fig:approx_illustration} illustrates the effect of the three models when applied on an image.

\subsection{Numerical experiments}
We used the developed sequential model correction method to restore images distorted by the nonlinear diffusion and curvature flow operators. For the fixed method we used a linear diffusion operator (constant diffusivity) to approximate the nonlinear operators in both cases. For the adaptive method we used the first-order Taylor expansion centered at the current iterate as an approximation. We set the contrast parameter in Perona-Malik filter to $\kappa = 0.1$ and the threshold parameter in the curvature flow model to $\psi = 0.001$. The step size $\delta t$ was set to $0.1$ and number of steps to 15 in all time-stepping algorithms. Either $3$ \% of Gaussian noise or $4$ \% of impulse noise was added to the distorted image. We numerically evaluated the reconstruction quality over a batch of 32 images, measuring the correspondence to ground truth images using peak signal-to-noise ratio (PSNR) and structural similarity index measure (SSIM) \cite{ssimwang}. We also investigated empirically the convergence of the methods by computing the value of the data fidelity at each iteration of the sequence. 

We compared our methods to the case where the approximation error was not dealt with, i.e., the reconstruction was computed using only the approximate model. We also compared our methods to AEM, where the approximation error is assumed to have a Gaussian distribution. The evaluation was done separately for the nonlinear diffusion and curvature flow models with Gaussian and impulse (salt and pepper) noise which assigns random pixels a value zero or one \cite{digital_image_processing}. For the Gaussian noise case we assumed squared $L^2$ data fidelity and for the impulse noise case we assumed $L^1$ data fidelity. In all experiments we used TV regularization which penalizes the absolute value of the image gradient, favoring piecewise constant reconstructions \cite{rof1992}. The regularization parameter was chosen separately for different methods by computing reconstructions for a test image with different values of $\lambda$ and selecting the value that yielded the highest SSIM.

\begin{algorithm}[!t]
    \caption{Sequential model correction with primal-dual method for $L^2/L^1$ data fidelity and TV regularization}
    \label{alg:seq}
    \begin{algorithmic}[1]
            \State Require $x_0, \lambda$
            \State $u_0 \gets x_0$, $\Bar{u}_0 \gets x_0$, $p_0 \gets 0$, $q_0 \gets 0$, $k \gets 0$, $\theta \gets 1$
            \While{sequence not converged} 
                \State $K_k \gets \|(\Tilde{A}_k, \nabla)\|_\mathrm{op}$
                \State $\tau_k \gets 1/K_k$
                \State $\sigma_k \gets 1/K_k$
                \State $\varepsilon(x_k) = A(x_k) - \Tilde{A}x_k$
                \State $t \gets 0$
                \While{primal-dual not converged}
                    \State $r_{t+1} = \Tilde{A}_k \Bar{u}_t - y + \varepsilon(x_{k})$
                    \State $p_{t+1} \gets (p_t + \sigma_k r_{t+1})/(1 + \sigma_k)$ \Comment{$L^2$}
                    \State $p_{t+1} \gets (p_t + \sigma_k r_{t+1})/\max(1_Y, |p_t + \sigma_k r_{t+1}|)$ \Comment{$L^1$}
                    \State $q_{t+1} \gets \lambda(q_t + \sigma_k \nabla \Bar{u}_t)/\max(\lambda 1_X, |q_t + \sigma_k \nabla \Bar{u}_t|)$
                    \State $u_{t+1} \gets u_t - \tau_k \Tilde{A}_k^* p_{t+1} + \tau_k \mathrm{div}(q_{t+1})$
                    \State $\Bar{u}_{t+1} \gets u_{t+1} + \theta (u_{t+1} - u_t)$
                    \State $t \gets t + 1$
                \EndWhile
            \State Perform line search to find $\delta_k$
            \State $x_{k+1} \gets (1 - \delta_k)\Bar{u}_t + \delta_k x_k$
            \State $k \gets k + 1$
            \State $\Bar{u}_0 \gets \Bar{u}_t$, $u_0 \gets u_t$, $p_0 \gets p_t$, $q_0 \gets q_t$
            \EndWhile
            \State \textbf{return} $x_k$
    \end{algorithmic}
\end{algorithm}

\subsection{Implementation}
The algorithms for sequential model correction and AEM were implemented with Python. The algorithm for sequential model correction with squared $L^2$- or $L^1$-data fidelity and TV regularization is presented in Algorithm \ref{alg:seq}. The subproblems requiring convex optimization were solved with primal-dual methods \cite{chambollepock2011, chambollepockappl}. The Jacobians for local linear approximations were computed via Jacobian-vector and vector-Jacobian products using the autograd library in Pytorch \cite{pytorch}. The computations were performed with a workstation with two $2.20$ GHz processors and Nvidia Quadro P4000 GPU.

\subsection{Data}
For testing our methods we used the FFHQ dataset \cite{dataset}. The dataset consists of 70000 color images of aligned and cropped faces of size $1024^2$. The data was preprocessed by converting the images to grayscale and downsampling to size $256^2$ to make the computations feasible. The dataset was also used to compute the mean vector and covariance matrix of the approximation error required by AEM.

\section{Results} \label{sec:results}

\begin{table}[!t]
\begin{center}
\begin{tabular}{|c|c|c|c|c|}
\hline
    & \multicolumn{2}{|c|}{NLD (Gaussian)} & \multicolumn{2}{|c|}{NLD (Impulse)} \\
    \cline{2-5}
    & PSNR & SSIM & PSNR & SSIM \\   
    \hline 
    No correction & 26.96 (2.58) & 0.87 (0.039) & 24.58 (2.61) & 0.87 (0.042)\\
    Fixed seq. & 28.48 (2.25) & 0.88 (0.037) & 30.05 (2.53) & 0.93 (0.022)\\
    Adaptive seq. & 29.32 (1.72) & 0.89 (0.031) & 29.62 (2.08) & 0.93 (0.020)\\
    AEM & 27.11 (2.17) & 0.85 (0.048) & - & -\\
    Data & 27.40 (0.92) & 0.80 (0.020) & 18.77 (0.27) & 0.48 (0.033) \\
\hline
\end{tabular}
\captionof{table}{Average (standard deviation) peak signal-to-noise ratio (PSNR) and structural similarity index measure (SSIM) over reconstructions of 32 images distorted by the nonlinear diffusion operator. The results are presented separately for Gaussian noise and impulse noise.}\label{table:nld}
\end{center}
\end{table}

\begin{table}[h!]
\begin{center}
\begin{tabular}{|c|c|c|c|c|}
\hline
    & \multicolumn{2}{|c|}{CF (Gaussian)} & \multicolumn{2}{|c|}{CF (Impulse)} \\
    \cline{2-5}
    & PSNR & SSIM & PSNR & SSIM \\   
    \hline 
    No correction & 27.39 (2.38) & 0.87 (0.042) & 26.21 (2.32) & 0.88 (0.038)\\
    Fixed seq. & 28.01 (2.36) & 0.88 (0.042) & 29.97 (2.43) & 0.93 (0.025)\\
    Adaptive seq. & 28.18 (2.26) & 0.88 (0.040) & 30.46 (2.86) & 0.93 (0.029)\\
    AEM & 26.42 (2.04) & 0.85 (0.046) & - & -\\
    Data & 26.08 (1.65) & 0.77 (0.027) & 18.48 (0.46) & 0.46 (0.024) \\
\hline
\end{tabular}
\captionof{table}{Average (standard deviation) peak signal-to-noise ratio (PSNR) and structural similarity index measure (SSIM) over reconstructions of 32 images distorted by the curvature flow operator. The results are presented separately for Gaussian noise and impulse noise.}\label{table:cf}
\end{center}
\end{table}

\begin{table}[h!]
\begin{center}
    \begin{tabular}{|c|c|c|c|c|}
\hline
         & \multicolumn{4}{|c|}{Time in seconds} \\
         \cline{2-5}
         & NLD (Gaussian) & NLD (Impulse) & CF (Gaussian) & CF (Impulse) \\
         \hline
        No correction & 10.84 (2.04) & 28.06 (5.57) & 9.05 (1.26) & 26.35 (4.70)\\
        Fixed seq. & 18.36 (4.56) & 230.24 (71.47) & 13.74 (3.02) & 123.27 (32.47)\\
        Adaptive seq. & 153.75 (84.22) & 373.76 (123.52) & 71.58 (15.76) & 301.29 (84.81)\\
        AEM & 91.51 (15.43) & - & 237.48 (26.98) & -\\
\hline
    \end{tabular}
    \captionof{table}{Average computation times (standard deviations) of different methods over 32 images.}\label{table:comptimes}
\end{center}
\end{table}

We start by examining the quantitative results. Tables \ref{table:nld} and \ref{table:cf} show the averaged PSNR and SSIM for reconstructions over a batch of 32 images with the NLD and CF models, respectively. The general trend is clear: using the fixed sequential correction method gives a clear improvement compared to not correcting the model at all. With NLD model and Gaussian noise the difference in PSNR is about 1.5 dB and 0.01 in SSIM. The difference with impulse noise is even more pronounced with about 5.5 dB in PSNR and 0.06 in SSIM. With NLD model and Gaussian noise the adaptive method gives a slight improvement over the fixed method. Interestingly, for impulse noise the effect is the opposite, a matter we discuss in Section \ref{sec:discussion}. For the CF model the difference between fixed and adaptive methods is negligible. The performance of AEM is slightly lower than using no correction with both models.

We then take a qualitative look at the reconstructed images. Figure \ref{fig:nld_gaussian} shows reconstructions with different correction methods for the NLD model approximated with a LD model with Gaussian noise. Without any correction the edges in the image are oversharpened since the LD model assumes higher diffusivity over the edges. The adaptive method is able to recover a bit more details than the fixed method, for example in the earring. Qualitatively the reconstruction with AEM is quite close to the reconstruction with fixed approximation, and at some parts a bit more blurry. Figure \ref{fig:cf_gaussian} shows the same reconstructions for the CF model. Here the difference between fixed and adaptive methods is not so clear. The AEM reconstruction is clearly the worst with the eyes not properly recovered. Reconstructions for the NLD model and impulse noise are shown in Figure \ref{fig:nld_impulse}. Here, correcting the model greatly increases the ability to separate details in the image. The fixed correction creates some wavy artefacts near the edges. The adaptive correction takes care of the artefacts and the text on the hat is more clearly visible. However, there is still some noise left in some parts of the image. We will discuss this phenomenon in Section \ref{sec:discussion}. Figure \ref{fig:cf_impulse} shows the reconstructions with CF model and impulse noise. Again, correcting the model increases the amount of details in the image. Again, the fixed correction introduces some artefacts at the edges while the adaptive correction takes care of the artefacts.

\begin{figure*}[!t]
    \centering
    \includegraphics[width=0.9\linewidth]{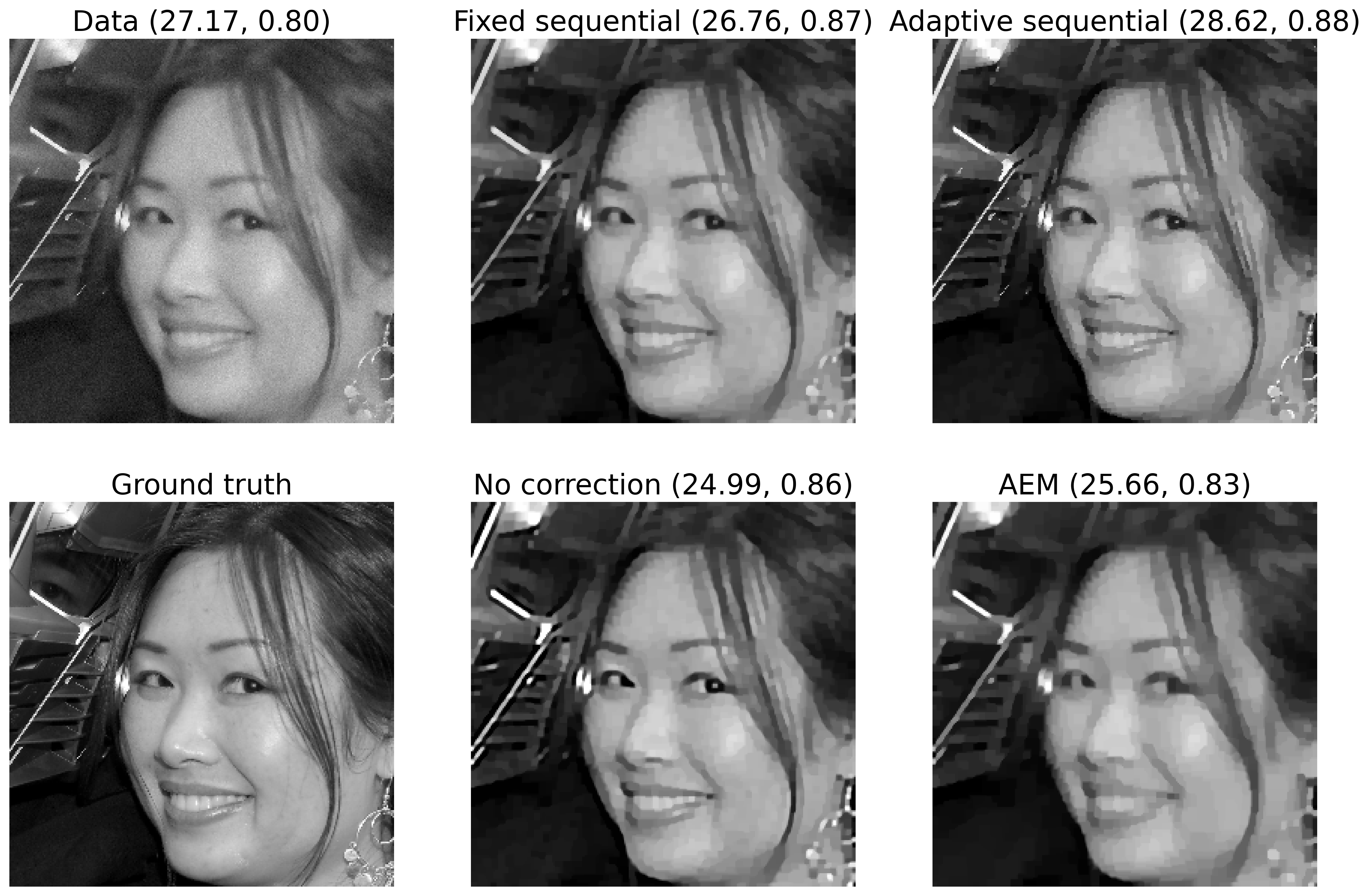}
    \caption{Reconstructions for the nonlinear diffusion model approximated with linear diffusion model with $3 \%$ Gaussian noise. The numbers in parentheses indicate peak signal to noise ratio (dB) and structural similarity index measure.}
    \label{fig:nld_gaussian}
\end{figure*}

\begin{figure*}[!t]
    \centering
    \includegraphics[width=0.9\linewidth]{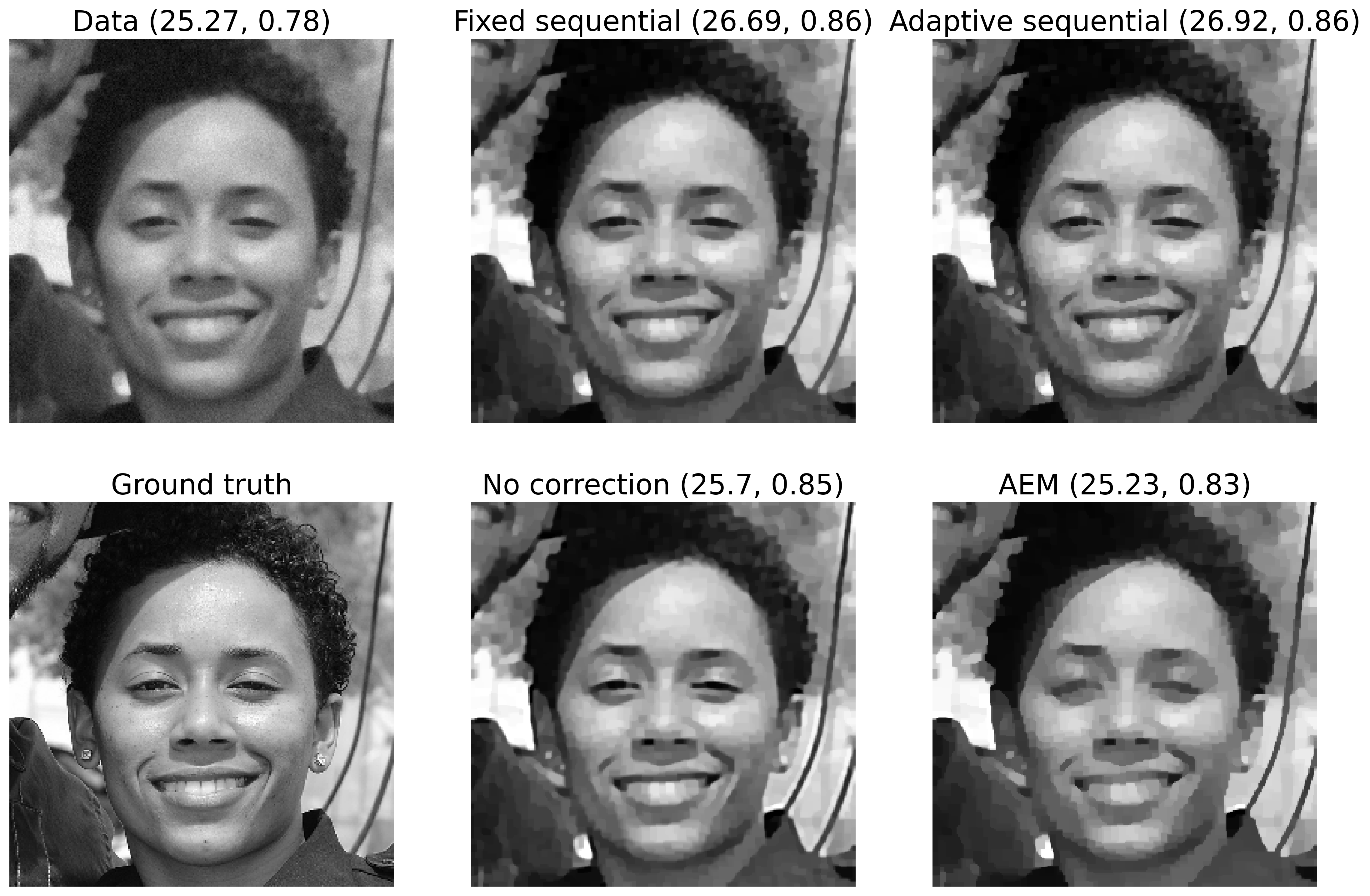}
    \caption{Reconstructions for the curvature flow model approximated with linear diffusion model with $3 \%$ Gaussian noise. The numbers in parentheses indicate peak signal to noise ratio (dB) and structural similarity index measure.}
    \label{fig:cf_gaussian}
\end{figure*}

\begin{figure*}[!t]
    \centering
    \includegraphics[width=0.9\linewidth]{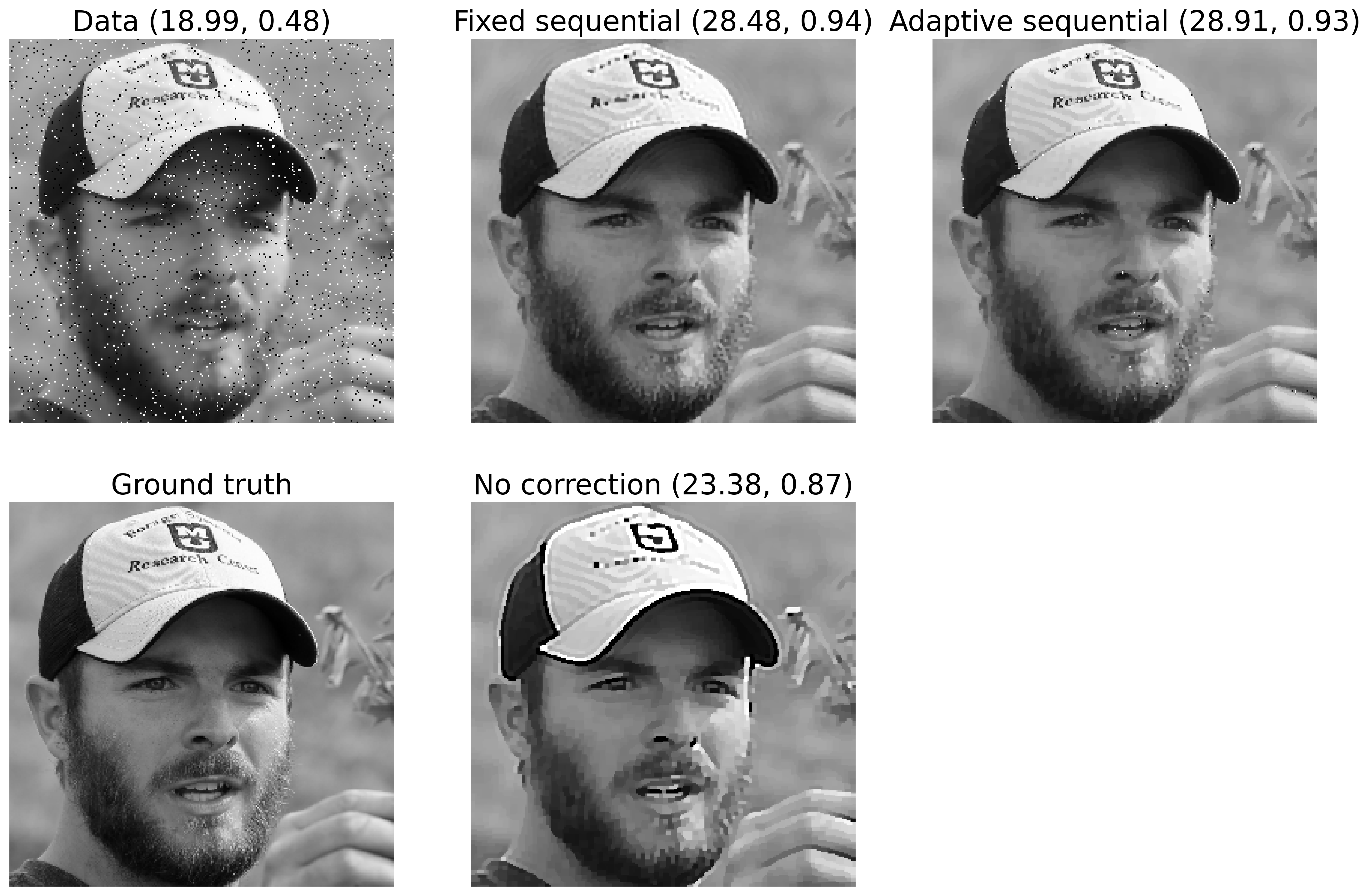}
    \caption{Reconstructions for the nonlinear diffusion model approximated with linear diffusion model with $4 \%$ salt and pepper noise. The numbers in parentheses indicate peak signal to noise ratio (dB) and structural similarity index measure.}
    \label{fig:nld_impulse}
\end{figure*}

\begin{figure*}[!t]
    \centering
    \includegraphics[width=0.9\linewidth]{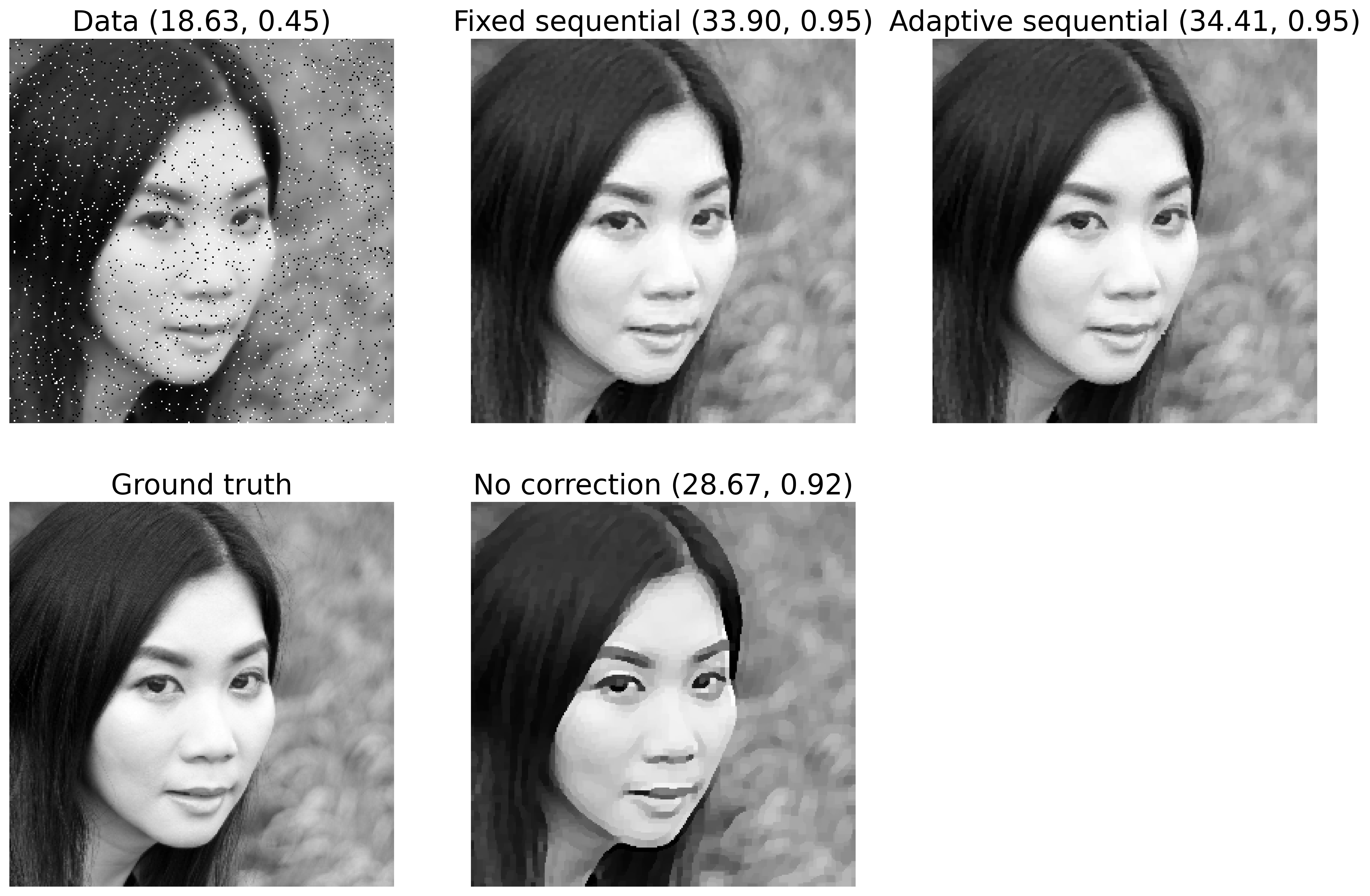}
    \caption{Reconstructions for the curvature flow model approximated with linear diffusion model with $4 \%$ salt and pepper noise. The numbers in parentheses indicate peak signal to noise ratio (dB) and structural similarity index measure.}
    \label{fig:cf_impulse}
\end{figure*}

Finally, we investigate the convergence of the sequences. Figures \ref{fig:convergence_l1} and \ref{fig:convergence_l2} show the evolution of the data fidelity $F(A(u), y)$ with respect to true model over sequence iterations. The behavior is as expected: the adaptive method converges to a lower value than the fixed method, experimentally confirming the analysis of Theorem \ref{thm:descend}, which stated that taking small enough steps in the sequence with adaptive approximation ends up decreasing the value of the objective function with respect to the true model. Furthermore, with the NLD model, the reconstructions with AEM obtain a slightly higher value of the data fidelity than with fixed approximation, while not using any correction gives the highest value. For the CF model the order of AEM and no correction is reversed. We can see from Table \ref{table:comptimes} that the fixed correction method offers a compromise between reconstruction quality and computational effort, compared to the adaptive method. The computation time is about one order of magnitude lower with the NLD model and Gaussian noise. The difference is less drastic with other models and noise types, due to slower convergence of the linear problem.

\begin{figure}[!h]
\centering
\begin{subfigure}{.39\textwidth}
    \centering
   \includegraphics[width=1\linewidth]{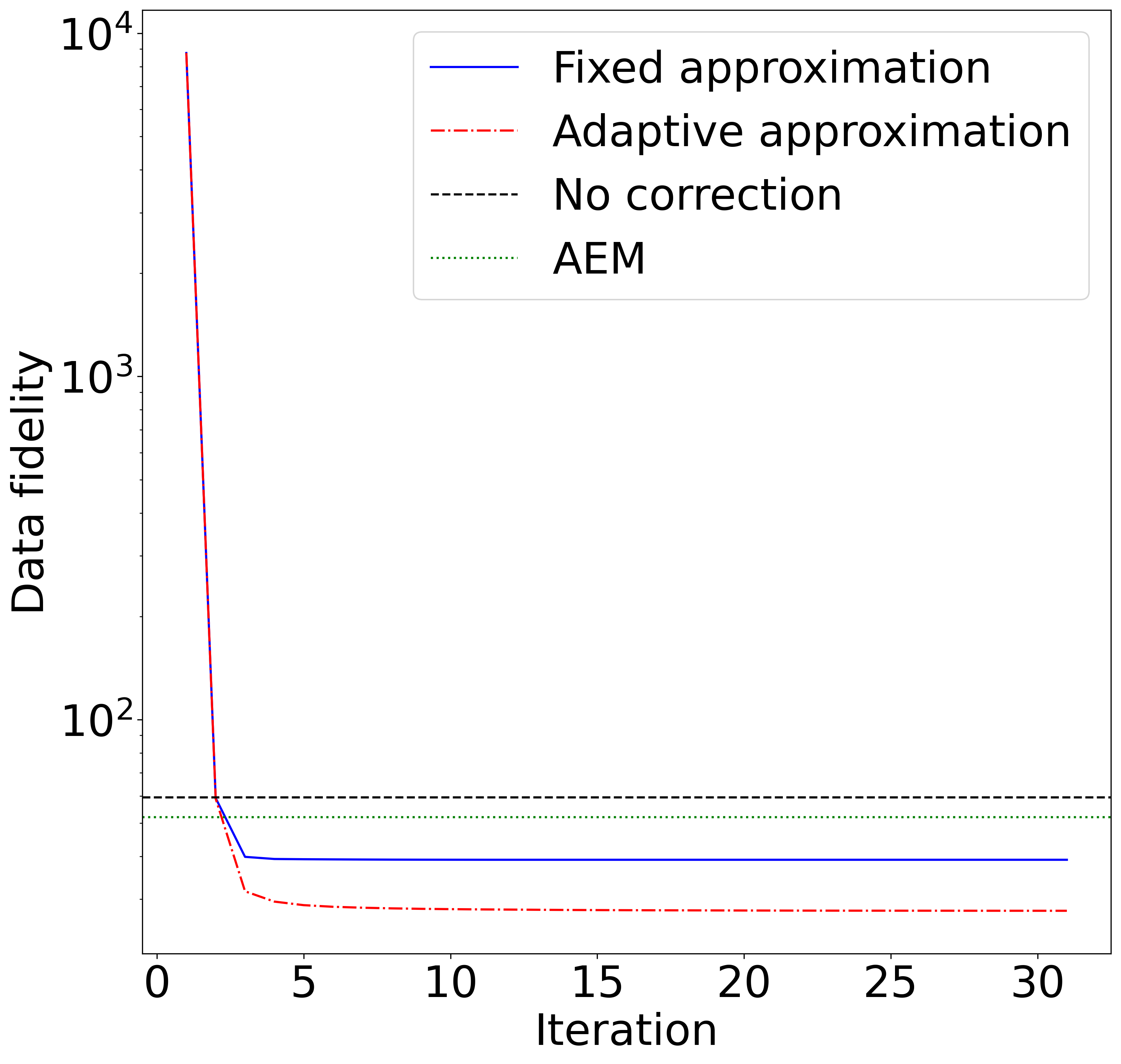}
   \caption{}
\end{subfigure}
\begin{subfigure}{.39\textwidth}
    \centering
   \includegraphics[width=1\linewidth]{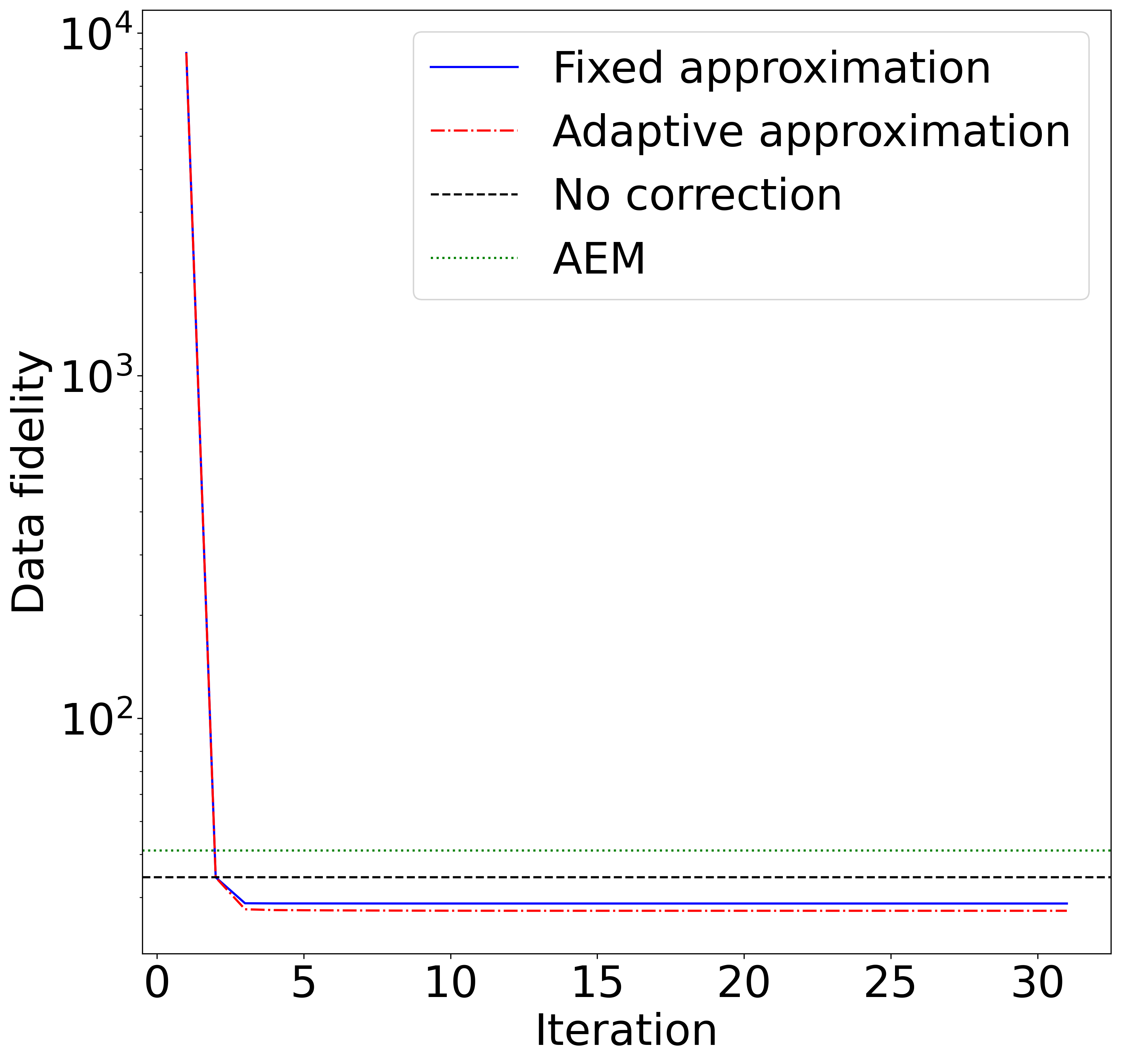}
   \caption{}
\end{subfigure}
\caption{Average of the $L_2$ data fidelity over reconstructions of 32 images with Gaussian noise for (a) nonlinear diffusion and (b) curvature flow models.}
\label{fig:convergence_l2}
\end{figure}

\begin{figure}[!h]
\centering
\begin{subfigure}{.39\textwidth}
    \centering
   \includegraphics[width=1\linewidth]{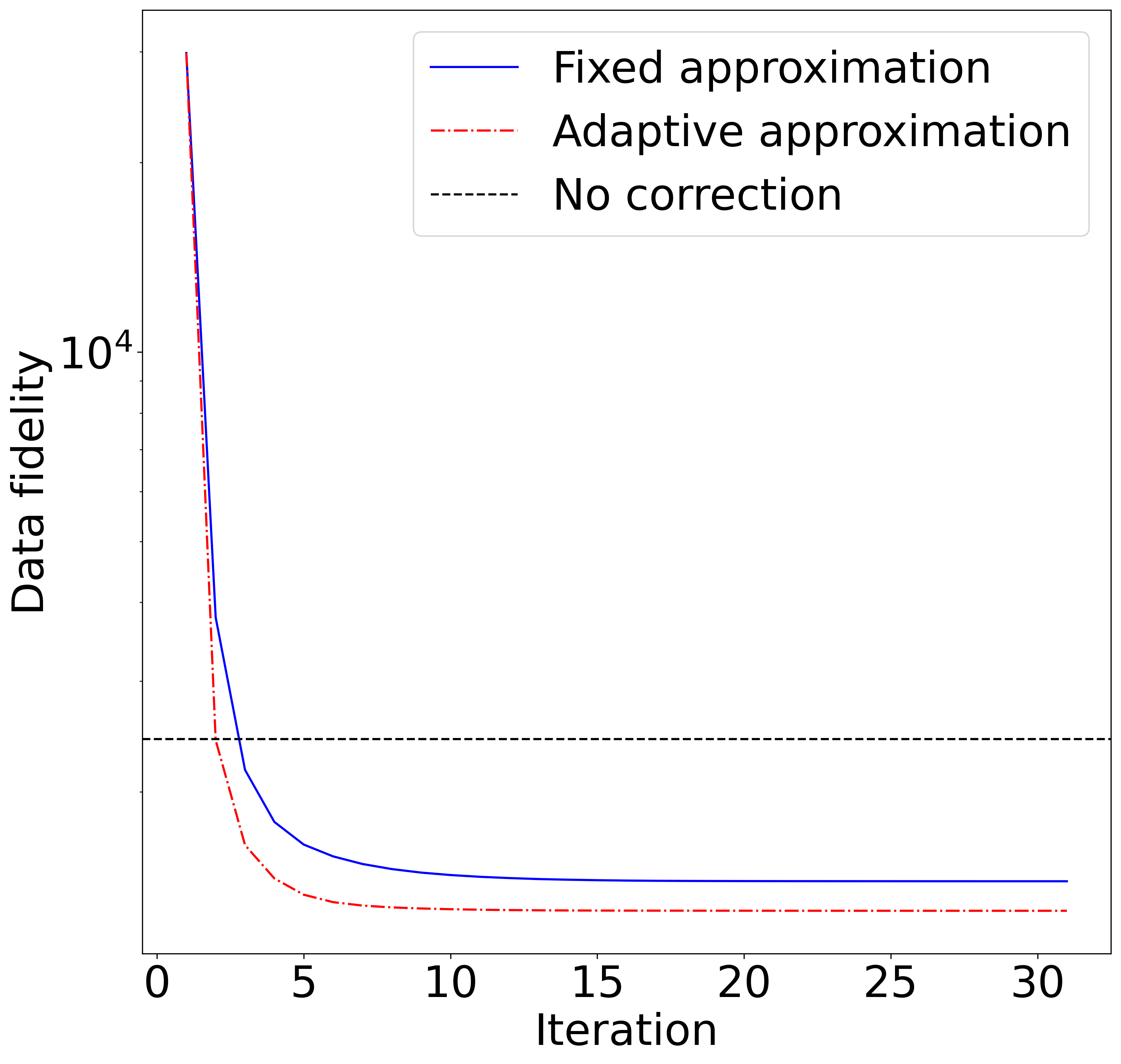}
   \caption{}
\end{subfigure}
\begin{subfigure}{.39\textwidth}
    \centering
   \includegraphics[width=1\linewidth]{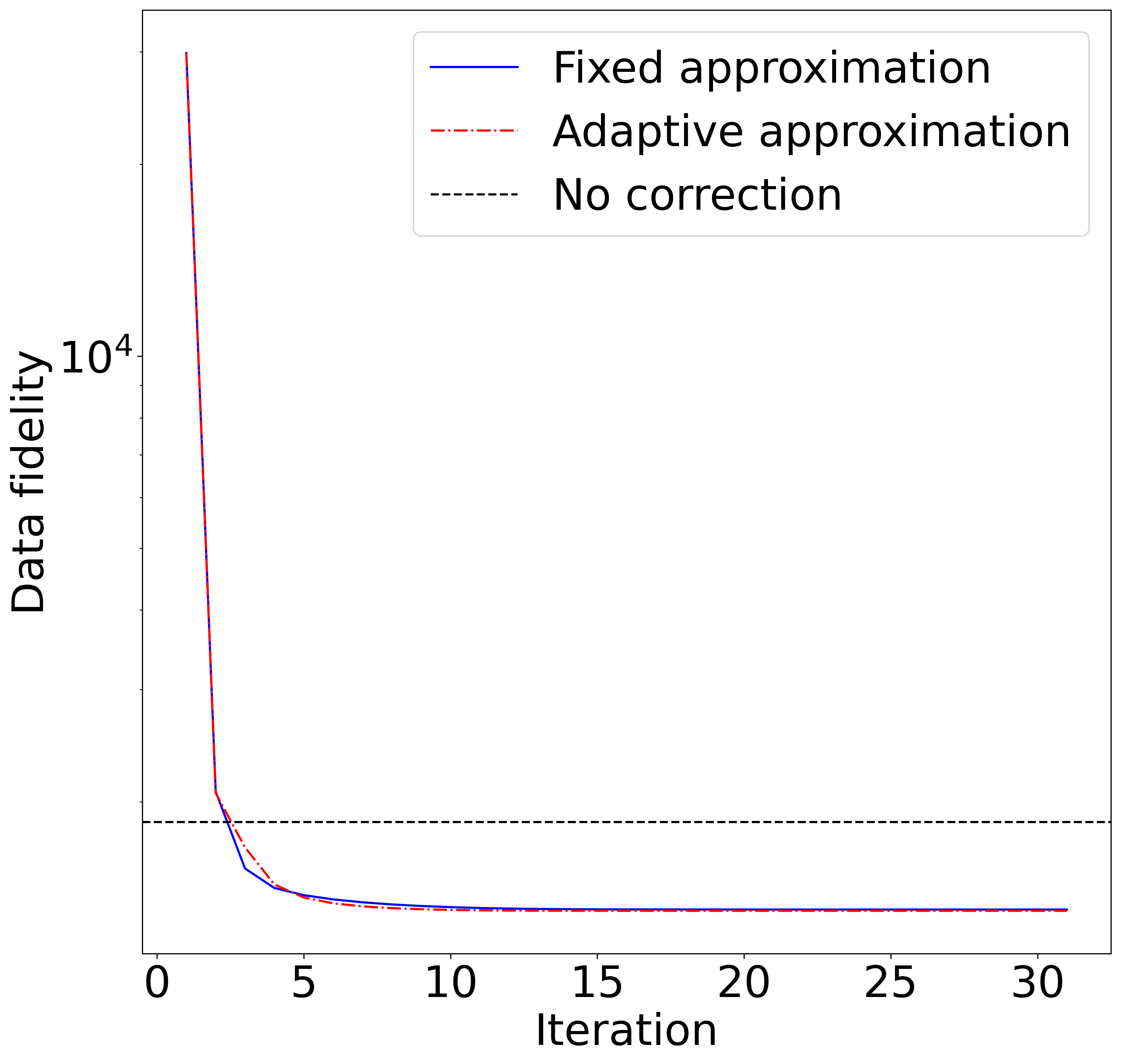}
   \caption{}
\end{subfigure}
\caption{Average of the $L_1$ data fidelity over reconstructions of 32 images with impulse noise for (a) nonlinear diffusion and (b) curvature flow models.}
\label{fig:convergence_l1}
\end{figure}

\begin{figure}[!h]
\centering
\begin{subfigure}[b]{0.8\textwidth}
   \includegraphics[width=1\linewidth]{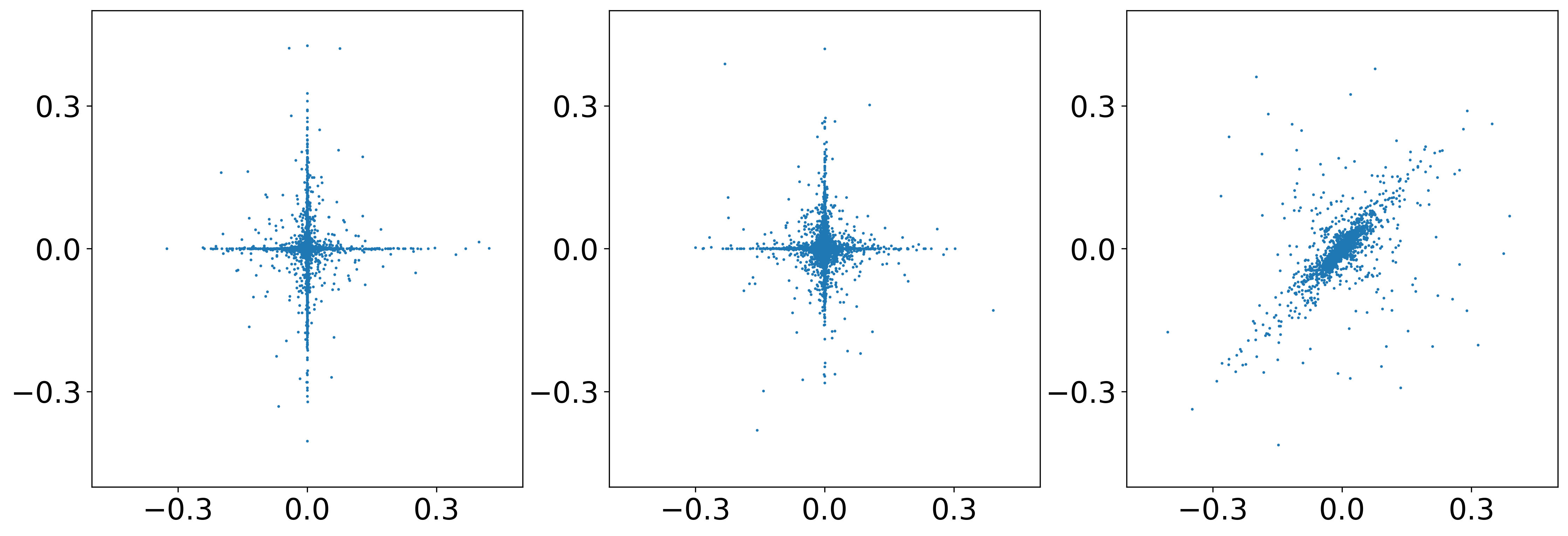}
   \caption{}
   \label{fig:Ng1} 
\end{subfigure}

\begin{subfigure}[b]{0.8\textwidth}
   \includegraphics[width=1\linewidth]{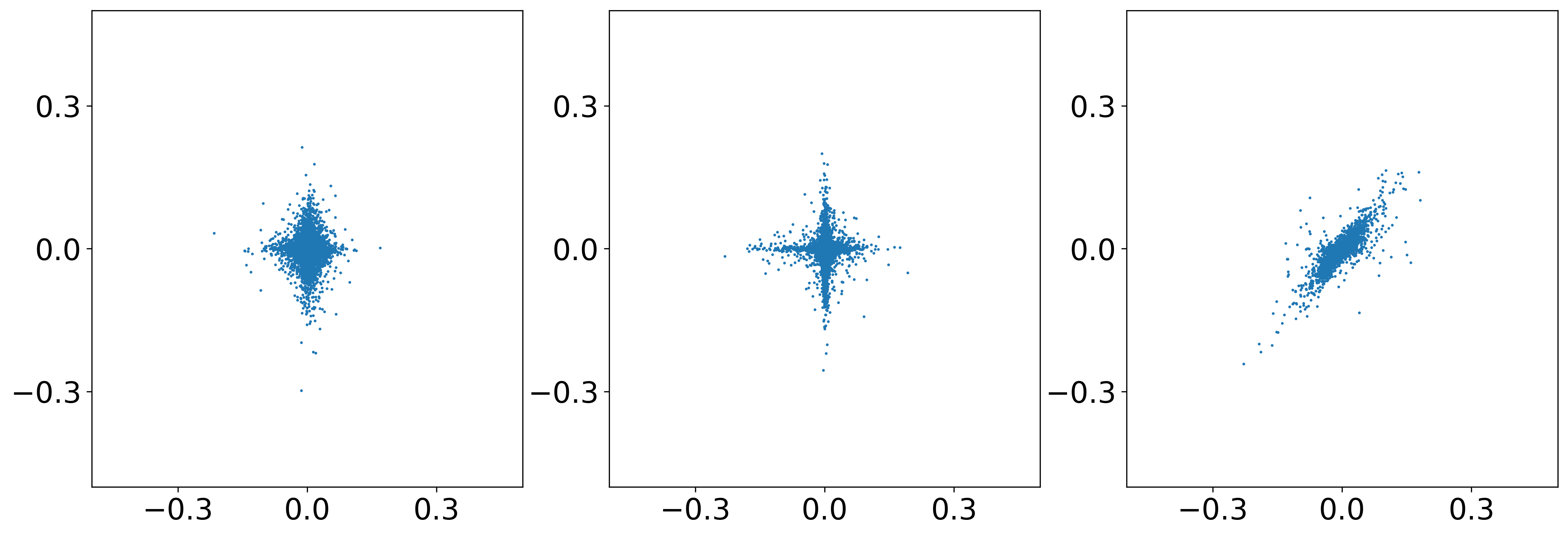}
   \caption{}
   \label{fig:Ng2}
\end{subfigure}
\caption{Scatterplots of some elements of the approximation error between (a) nonlinear and linear diffusion and (b) curvature flow and linear diffusion operators. These figures highlight non-Gaussianity in the approximation error and thus the need for a non-Gaussian correction.}
\label{fig:approxerrdist}
\end{figure}

\subsection{Discussion of results} \label{sec:discussion}
With the NLD model and impulse noise, the fixed method gives quantitatively better results than the adaptive method. This is mainly due to the noise left in the reconstruction (see Figure \ref{fig:nld_impulse}). While increasing the regularization parameter could get rid of the noise, the quality of the image would deteriorate in other locations, resulting in oversmoothing. A possible fix could be to consider other kinds of data fidelity models, such as Cauchy noise model \cite{cauchynoise}. It penalises outliers even less than the $L^1$ data fidelity which is based on the Laplace distribution. As we consider only convex data fidelity terms, we leave the Cauchy model to future studies as it is non-convex.

%The computation time with impulse noise is lower for the adaptive method than for the fixed method. This seems a bit counterintuitive since the adaptive method requires the computation of the Jacobian. The reason is the slower convergence of the fixed method (see Figure \ref{fig:convergence_l1}). The adaptive method converges in about 5--10 steps while the fixed method converges in 20--25 steps on average. The reason for slower convergence is in the regularization parameter. The average value of the regularization parameter giving the best SSIM for the adaptive method is 0.24 while for the fixed method it is 0.0051. As discussed in Section \ref{sec:conv_fixed}, the convergence of the fixed method depends on the regularization parameter: the higher it is the more likely the sequence is to converge. 
The computation time differences between fixed and adaptive methods are less pronounced for $L^1$ data fidelity, especially with the NLD model. A possible reason could be the slower convergence of the fixed method which converges in 15--20 iterations while the adaptive method converges in about 10--15 iterations on average. Another reason for the relative small differences in computation time is the efficiency of the computation of the Jacobian that comes from the autograd library in PyTorch utilizing GPU parallelization.

As discussed earlier, AEM is based on the Gaussianity assumption of the approximation error. Figure \ref{fig:approxerrdist} illustrates some of the pairwise distributions of the approximation error between the models used in this work. All in all, the distributions seem very non-Gaussian, with star-like shapes and outliers quite far from the center of the distribution. Even though the distribution is not Gaussian, the reconstructions with AEM with the NLD model achieve on average lower value of the data fidelity than not using any correction at all. It could be that the approximation error consists of multiple components of which some are Gaussian. AEM learns that part and leaves the rest of the error to be modeled as noise. We note that there exist also hierarchical constructions of the AEM, namely where the unknown is modeled as conditionally Gaussian (see e.g., \cite{calvetti_mesh_aem}). Hierarchical modeling of Gaussian variables brings some flexibility and adaptiveness to model and can help in recovering non-Gaussian structures. We note that recently proposed neural network-based correction methods are also capable of correcting non-Gaussian errors but require training data to use \cite{opcor2021lunz,smylnongaussian}.

\subsection{Comparison to other optimization techniques}
Adaptively approximating the nonlinear model with Taylor expansion revealed connections between model correction and optimization. With certain smoothness assumptions for the model and objective function, we proved that the sequence with adaptive approximation always decreases the objective function with respect to the exact model. We also showed that with quadratic objective functions the adaptive sequence corresponds to the classic Gauss-Newton method. As Gauss-Newton assumes that the data fidelity term is given by the $L^2$-norm, our approach is slightly more general than it.

The adaptive sequential correction method is also closely related to the majorization-minimization (MM) framework \cite{hunter2004tutorial}. In both methods, a convex surrogate of the original nonconvex function is constructed about the current iterate, and the surrogates are sequentially minimized. In MM, the surrogate is a majorizer, i.e., it is larger than the original function for every input, while this is not necessarily the case for our method. Furthermore, MM framework requires constructing the surrogate for each application separately and there are different methods for constructing it, whereas for our sequential method the surrogate follows naturally from the linearization of the model. However, MM is slightly more general than our method as we require convexity of the data fidelity and regularizer.

Trust-region methods are also related to our method \cite{kelley1999iterative}. They are sequential optimization methods that specify a ball of radius $\Delta$ (trust-region) about the current iterate. In the trust-region, the objective function is approximated with a quadratic function that agrees with the objective function up to the first derivative. The quadratic function is approximately minimized and either the minimizer is chosen as the next iterate or the radius of trust-region is reduced. The difference to our method is that the surrogate specified by the linearization is not necessarily quadratic.

Recent works have focused on extending the theory of primal-dual methods to allow nonlinear operators \cite{chen2021nonconvex,valkonen2014primal}. They are closely related to our work as they use different kinds of linearizations to convexify the problem. However, their approach is to modify the existing primal-dual algorithms for convex optimization to deal with the nonlinearity. This is fundamentally different to our method as we sequentially update the nonlinear term and use convex optimization to solve the sub-problem. Furthermore, our approach is not limited to primal-dual methods, any algorithm for convex optimization works.

\section{Conclusion} \label{sec:conclusion}
In this work we have proposed a strategy for correcting nonlinear models in the variational framework. We started from the observation that the conventional method for model correction that assumes a Gaussian distribution for the approximation error is not suitable for correcting non-Gaussian errors. These kind of errors arise especially when trying to correct nonlinear models with a linear approximation, as discussed in Section \ref{sec:aem}. The proposed strategy involves finding a linear approximation of the nonlinear model and solving the arising convex variational problem using the linear model. Updating the approximation error at the solution of the variational problem, the process is repeated until convergence. We investigated two different kinds of approximation, fixed and adaptive. The fixed approximation does not depend on the iteration number, making it computationally and conceptually simple to use, only requiring a few evaluations of the accurate nonlinear model. The sequence with the fixed approximation can be thought of as a fixed-point iteration, with simple conditions telling whether the sequence converges or not. Unfortunately, in practice it is rather difficult to tell if the conditions are fulfilled. Another possibility to ensure convergence is to terminate the sequence when the objective function with the exact operator can no longer be decreased. This involves evaluations of the correct model and might not be feasible if one evaluation is time consuming. In the case where the approximation is adaptive, we were able to draw connections between model correction and optimization literature. Specifically, if the approximation is chosen as a Taylor expansion, the adaptive sequence can be seen as an optimization method.

The connection to Gauss-Newton and MM methods could inspire further research to establish connections between model correction and nonlinear optimization. In many cases it is computationally prohibitive to differentiate a nonlinear operator. Here, other approximations that satisfy the convergence criterion of Theorem \ref{thm:descend} would be of further interest for future studies. Additionally, computationally cheap approximations of the derivative could be of interest, for instance by (learned) Quasi-Newton methods \cite{smyl2021efficient}.

Finally, we will consider the application of the sequential approximation for the use with other nonlinear PDE based inverse problems \cite{mozumder2021model,sahlstrom2020modeling}. Here, we believe that the fixed approximation without the need to differentiate the model could be of great computational advantage.

\section*{Acknowledgements} Much of the theory was developed during Arttu Arjas' visit to Marcelo Pereyra in Heriot-Watt University in Edinburgh. We also want to thank Simon Arridge for helpful discussions.

\bibliographystyle{siam}
\bibliography{seqbib.bib}

\end{document}